\documentclass[11pt, reqno]{article}
\usepackage{textcmds} 
\usepackage{amsmath, amssymb, amsfonts, amstext, verbatim, amsthm, mathrsfs}
\usepackage{microtype}
\usepackage[all]{xy}
\usepackage[modulo]{lineno}
\usepackage[dvipsnames]{xcolor}
\usepackage{aliascnt}
\usepackage{enumitem}
\usepackage{xspace}
\usepackage{amsfonts}
\usepackage{amssymb}
\usepackage{amsthm}
\usepackage{rotating}
\usepackage[margin=3.5cm]{geometry}
\usepackage{dsfont}
\usepackage{bm}
\usepackage{subfigure}
\usepackage{amsmath}
\usepackage{array}
\usepackage[all]{xy}
\usepackage{euscript}
\usepackage[T1]{fontenc}
\usepackage[bb=stix]{mathalfa}
\usepackage{nccmath}
\usepackage{marginnote}
\usepackage{stmaryrd}
\usepackage{youngtab}
\usepackage{tikz}
\usepackage{calligra}
\usepackage{tikz-cd}
\usepackage{blkarray}
\usepackage{chngcntr}
\usepackage{multicol}
\usepackage{cancel}
\usepackage{titlesec}
\usepackage{caption}
\usepackage[utf8]{inputenc}  
\usetikzlibrary{calc}
\usepackage{graphics,graphicx}  
\usepackage[backend=biber,style=alphabetic, citestyle=alphabetic, isbn=false, doi=false, url=false, maxbibnames=10,sorting=nyt]{biblatex}
\usepackage[colorlinks=true,linkcolor=black,citecolor=magenta,urlcolor=blue,citebordercolor={0 0 1},urlbordercolor={0 0 1},linkbordercolor={0 0 1}]{hyperref} 

\tikzset{node distance=3cm, auto}

\makeatother

\newtheorem{thm}{Theorem}[section]
\newtheorem{thmlet}{Theorem}

\newtheorem{cor}[thm]{Corollary}
\newtheorem{corlet}[thmlet]{Corollary}

\newtheorem{prop}[thm]{Proposition}

\newtheorem{lemma}[thm]{Lemma}

\newtheorem{problem}[thm]{Problem}

\newtheorem{definition}[thm]{Definition}

\newtheorem{notation}[thm]{Notation}

\theoremstyle{remark}

\numberwithin{equation}{subsection} 
\numberwithin{figure}{subsection}
\numberwithin{table}{subsection}
\numberwithin{thm}{subsection}

\newtheoremstyle{customremark}
{3pt}
{3pt}
{}
{}
{\bfseries}
{.}
{.5em}
{}
\theoremstyle{customremark}
\newtheorem{rmk_no_diamond}[thm]{Remark}
\newenvironment{rmk}{\begin{rmk_no_diamond} } {\hfill$\Diamond$ \end{rmk_no_diamond}}
\newtheorem{example_no_diamond}[thm]{Example}
\newenvironment{example}{\begin{example_no_diamond} } {\hfill$\Diamond$ \end{example_no_diamond}}
\newtheorem{convention_no_diamond}[thm]{Convention}
\newenvironment{convention}{\begin{convention_no_diamond} } {\hfill$\Diamond$ \end{convention_no_diamond}}

\newenvironment{psmallmatrix}
  {\left(\begin{smallmatrix}}
  {\end{smallmatrix}\right)}

 \usepackage[draft]{say}
\newcommand{\fib}{{\op{Fib}}}
\newcommand{\bP}{\mathbf{P}}
\newcommand{\len}{\op{len}}
\newcommand{\out}{\op{out}}

\newcommand{\wt}{\widetilde}
\newcommand{\db}{{\frak{p}}}
\newcommand{\ra}{\rightarrow}
\newcommand{\pp}{{\mathfrak{p}}}

\newcommand{\R}{\mathbb{R}}
\newcommand{\NI}{{\noindent}}
\newcommand{\Z}{\mathbb{Z}}
\newcommand{\C}{\mathbb{C}}
\newcommand{\E}{{\mathbb{E}}}
\newcommand{\ka}{\kappa}
\newcommand{\MM}{{\mathbb{M}}}
\newcommand{\Ddiv}{D}
\newcommand{\de}{\delta}
\newcommand{\CP}{\mathbb{CP}}
\newcommand{\calN}{\mathcal{N}}
\newcommand{\calB}{\mathcal{B}}
\newcommand{\calD}{{\mathcal{D}}}
\newcommand{\op}[1]{{\operatorname{#1}}}
\renewcommand{\ll}{\llbracket}
\newcommand{\rr}{\rrbracket}

\definecolor{darkmagenta}{rgb}{0.55, 0.0, 0.55}
\newcommand{\hl}[1] {{\boldmath\textbf{{\color{darkmagenta}#1}}}}
\newcommand{\frakd}{\frak{d}}
\newcommand{\aut}{\op{Aut}}
\newcommand{\NN}{\mathbb{N}}
\newcommand{\lan}{\langle}
\newcommand{\ran}{\rangle}
\newcommand{\ga}{\gamma}
\renewcommand{\setminus}{\smallsetminus}
\renewcommand{\1}{\mathbb{1}}
\newcommand{\ks}{{\operatorname{{\mathsf{S}}}}}
\newcommand{\nn}{\frak{n}}
\newcommand{\mm}{\frak{m}}
\newcommand{\ann}{\op{Ann}}
\renewcommand{\min}{\op{min}}
\newcommand{\ind}{\op{ind}}
\newcommand{\calM}{\mathcal{M}}

\newcommand{\nil}{\varnothing}

\newcommand{\std}{\op{std}}
\newcommand{\om}{\omega}
\newcommand{\veca}{{\vec{a}}}
\newcommand{\vecaprime}{\veca\hspace{.1em}'}
\newcommand{\hooksymp}{\overset{s}\hookrightarrow}
\newcommand{\degmin}{d_\op{min}}
\newcommand{\sss}{\vspace{2.5 mm}}
\renewcommand{\root}{{\hspace{-.35mm}\scalebox{1.15}{\text{\scalebox{.8}[.5]{\textsurd}}}}}
\newcommand{\calDroot}{\calD_\root}

\newcommand{\Q}{\mathbb{Q}}

\newcommand{\be}{\beta}
\newcommand{\ovl}{\overline}

\newcommand{\lineseg}[1]{\hspace{-.5em}\begin{tikzcd}[ampersand replacement=\&, column sep=8ex ] {} \arrow[dash,"{#1}"]{r} \& {} \end{tikzcd}\hspace{-.5em}}

\newcommand{\acc}{\op{acc}}

\newcommand{\bb}{\mathfrak{b}}

\newcommand{\bl}{{\op{Bl}}}

\newcommand{\Si}{\Sigma}

\newcommand{\calS}{\mathcal{S}}
\newcommand{\cone}{\op{Cone}}
\newcommand{\tor}{{\op{tor}}}
\newcommand{\F}{\mathbb{F}}
\renewcommand{\wp}{\mathfrak{W}}
\DeclareMathAlphabet{\mathcalligra}{T1}{calligra}{m}{n}
\newcommand{\tormod}{\mathcal{T}}

\newcommand{\Ndiv}{\mathcal{N}}

\newcommand{\al}{\alpha}
\newcommand{\area}{\op{area}}
\newcommand{\Rsym}{R}
\newcommand{\Ssym}{S}
\newcommand{\Sset}{\mathscr{S}}

\newcommand{\trap}{
    \raisebox{-0.2ex}{\begin{tikzpicture}[scale=0.08]
        \draw[thick] (0,0) -- (2,0) -- (2,1) -- (0,2) -- cycle;
    \end{tikzpicture}}
}
\newcommand{\trapdom}{X_{\trap}}
\newcommand{\calK}{\mathcal{K}}
\newcommand{\FF}{F}

\newcommand{\mmvec}{{\vec{\mm}}}
\newcommand{\bPvec}{{\vec{\bP}}}

\newcommand{\coef}{\op{Coef}}
\newcommand{\for}{\op{pr}}
\newcommand{\ff}{\mathbb{f}}
\newcommand{\nicequot}[2]{#1 \left/ \kern-0.3em\right. #2}
\newcommand{\mmsubN}{\mm_{\scriptscriptstyle \calN}}
\newcommand{\inn}{\op{in}}
\newcommand{\uvl}{\underline}

\newcommand{\Tsym}{T}
\newcommand{\wpcount}{N}
\newcommand{\si}{\sigma}

\newcommand{\gapac}{\mathfrak{g}}
\newcommand{\TT}{\mathbb{T}}
\newcommand{\wh}{\widehat}

\makeatletter
\newcommand{\dashover}[2][\mathop]{#1{\mathpalette\df@over{{\dashfill}{#2}}}}
\newcommand{\fillover}[2][\mathop]{#1{\mathpalette\df@over{{\solidfill}{#2}}}}
\newcommand{\df@over}[2]{\df@@over#1#2}
\newcommand\df@@over[3]{%
  \vbox{
    \offinterlineskip
    \ialign{##\cr
      #2{#1}\cr
      \noalign{\kern1pt}
      $\m@th#1#3$\cr
    }
  }%
}
\newcommand{\dashfill}[1]{%
  \kern-.5pt
  \xleaders\hbox{\kern.5pt\vrule height.4pt width \dash@width{#1}\kern.5pt}\hfill
  \kern-.5pt
}
\newcommand{\dash@width}[1]{%
  \ifx#1\displaystyle
    2pt
  \else
    \ifx#1\textstyle
      1.5pt
    \else
      \ifx#1\scriptstyle
        1.25pt
      \else
        \ifx#1\scriptscriptstyle
          1pt
        \fi
      \fi
    \fi
  \fi
}
\newcommand{\solidfill}[1]{\leaders\hrule\hfill}

\newcommand{\oset}[3][0ex]{%
  \mathrel{\mathop{#3}\limits^{
    \vbox to#1{\kern-2\ex@
    \hbox{$\scriptstyle#2$}\vss}}}}

\newcounter{countitems}
\newcounter{nextitemizecount}
\newcommand{\setupcountitems}{%
  \stepcounter{nextitemizecount}%
  \setcounter{countitems}{0}%
  \preto\item{\stepcounter{countitems}}%
}
\makeatletter
\newcommand{\computecountitems}{%
  \edef\@currentlabel{\number\c@countitems}%
  \label{countitems@\number\numexpr\value{nextitemizecount}-1\relax}%
}
\newcommand{\nextitemizecount}{%
  \getrefnumber{countitems@\number\c@nextitemizecount}%
}
\newcommand{\previtemizecount}{%
  \getrefnumber{countitems@\number\numexpr\value{nextitemizecount}-1\relax}%
}
\makeatother    
{\end{itemize}%
\unskip\computecountitems\ifnumcomp{\previtemizecount}{>}{3}{\end{multicols}}{}}

\makeatother

\title{Sesquicuspidal curves, scattering diagrams, and symplectic nonsqueezing}

\author{Dusa McDuff and Kyler Siegel\thanks{K.S. is partially supported by NSF grant DMS-2105578}}

\begin{document}

\maketitle

\begin{abstract}
We solve the stabilized symplectic embedding problem for four-dimensional ellipsoids into the four-dimensional round ball. 
The answer is neatly encoded by a piecewise smooth function which exhibits a phase transition from an infinite Fibonacci staircase to an explicit rational function related to symplectic folding.
Our approach is based on a bridge between quantitative symplectic geometry and singular algebraic curve theory, and a general framework for approaching both topics using scattering diagrams. 
In particular, we construct a large new family of rational algebraic curves in the complex projective plane with a $(p,q)$ cusp singularity, many of which solve the classical minimal degree problem for plane curves with a prescribed cusp.
A key role is played by the tropical vertex group of Gross--Pandharipande--Siebert and ideas from mirror symmetry for log Calabi--Yau surfaces.
Many of our results also extend to other target spaces, e.g. del Pezzo surfaces and more general rational surfaces.
\end{abstract}

\tableofcontents

\section{Introduction}\label{sec:intro}

Since Gromov's discovery of his famous nonsqueezing theorem in \cite{gromov1985pseudo}, a primary goal of quantitative symplectic geometry has been to put explicit nontrivial restrictions on Hamiltonian diffeomorphisms.
In particular, we have the following central problem:
\begin{problem}[ellipsoid embedding problem]\label{prob:REEP}
For which $\veca,\vecaprime \in \R_{>0}^n$ does there exist a symplectic embedding $E(\veca) \hooksymp E(\vecaprime)$?
\end{problem} 
\NI Here $E(\veca) := \left\{\pi\sum\limits_{i=1}^n \tfrac{(x_i^2+y_i)^2}{a_i}  \leq 1\right\}$ denotes the symplectic ellipsoid in $\R^{2n}$ with area factors $\veca = (a_1,\dots,a_n) \in \R_{>0}^n$, endowed with the restriction of the standard symplectic form $\om_\std = \sum\limits_{i=1}^n dx_i \wedge dy_i$. By \hl{symplectic embedding} we mean a smooth embedding which pulls back the symplectic form on the target to that of the source (this is equivalent to the existence of a Hamiltonian diffeomorphism $\Phi: \R^{2n} \ra \R^{2n}$ satisfying $\Phi(E(\veca)) \subset E(\vecaprime)$ -- see \cite[\S4.4]{Schlenk_old_and_new}).

In dimension $4$, the monotonicity properties of embedded contact homology inspired Hofer to
 conjecture a necessary and sufficient condition for the existence of such an embedding in dimension
$2n=4$, and this was established in one direction by Hutchings~\cite{Hutchings_quantitative_ECH} and in the other by McDuff~\cite{McDuff_Hofer_conjecture}.  Using this, 
a complete solution to Problem~\ref{prob:REEP} with target space the round ball $E(\vecaprime) = E(1,1) =: B^4$
was worked out explicitly by McDuff--Schlenk in \cite{McDuff-Schlenk_embedding}, building on various works \cite{mcduff2009symplectic,mcduff1994symplectic,mcduff_from_def_to_iso,biran1997symplectic,li2001uniqueness,li_li_2002} and with input from Seiberg--Witten theory. 
At present, what happens in higher dimensions is so little understood that there is no  conjecture as to the general answer.  However,
there is an intriguing ``stabilized'' regime with $a_3,\dots,a_n \gg a_1,a_2$ and $a_3',\dots,a_n' \gg a_1',a_2'$ which appears to serve as a bridge between four dimensions and higher dimensions. Using work by \cite{Pelayo-Ngoc_hofer_question} that allows one to pass from embeddings of a compact ellipsoid with $a_3,\dots,a_n \gg a_1,a_2$ to embeddings of the noncompact domain $E(a_1,a_2) \times \R^{2N}$, the question can be  formulated as follows: 
\begin{problem}[stabilized ellipsoid embedding problem]\label{prob:SEEP}
 For which $a_1,a_2,a_1',a_2' \in \R_{>0}$ and $N \in \Z_{\geq 1}$ does there exist a symplectic embedding $E(a_1,a_2) \times \R^{2N} \hooksymp E(a_1',a_2') \times \R^{2N}$?
\end{problem}

In this paper, we give a complete solution to the stabilized ellipsoid embedding problem in the case that the target is the stabilized round ball (i.e. $a_1'=a_2'$), along with various other target spaces.
In other words, we compute the \hl{stabilized ellipsoid embedding function} $c_{B^4 \times \R^{2N}}$, where for any $a \in \R_{\geq 1}$, $N \in \Z_{\geq 1}$, and symplectic four-manifold $X^4$ we put
\begin{align}
c_{X^4 \times \R^{2N}}(a) := \inf\left\{\mu\;|\; E(\tfrac{1}{\mu},\tfrac{a}{\mu}) \times \R^{2N} \hooksymp X^4 \times \R^{2N}\right\}.
\end{align}
Let $\fib_1 = 1, \fib_2 = 1, \fib_{k+2} = \fib_k + \fib_{k+1}$ denote the Fibonacci numbers, 
and put $\al_k := \tfrac{\fib_{2k+1}^2}{\fib_{2k-1}^2}$ and $\be_k := \tfrac{\fib_{2k+3}}{\fib_{2k-1}}$ for $k \in \Z_{\geq 1}$, so that we have
\begin{align*}
\al_0 := 1 < \be_0 := 2 < \al_1 = 4 < \be_1 = 5 < \al_2 = \tfrac{25}{4} < \be_2 = \tfrac{13}{2} < \al_3 = \tfrac{169}{25} < \cdots
\end{align*}
and $\lim\limits_{k \rightarrow \infty}\al_k = \lim\limits_{k \rightarrow \infty}\be_k = \tau^4 := \tfrac{7+3\sqrt{5}}{2} \approx 6.85$ (here $\tau$ is the golden ratio). 
\begin{thmlet}\label{thmlet:main_SEEP}
For any $N \in \Z_{\geq 1}$, the stabilized ellipsoid embedding function of the round four-ball $B^4$ is given by:
\begin{align*}
c_{B^4 \times \R^{2N}}(a) = 
\begin{cases} \frac{1}{\sqrt{\al_k}} \cdot a & \text{if}\;\; a \in [\al_k,\be_k] \text{ for some }k \in \Z_{\geq 0}\\ 
\sqrt{\al_{k+1}} & \text{if}\;\; a \in [\be_k,\al_{k+1}] \text{ for some }k \in \Z_{\geq 0}\\ 
\frac{3a}{a+1} & \text{if}\;\; a \in [\tau^4,\infty).
\end{cases}
\end{align*}
\end{thmlet}
\NI See Figure~\ref{fig:stab_fib} for an illustration. 
  \begin{figure}
  \includegraphics[scale=.8]{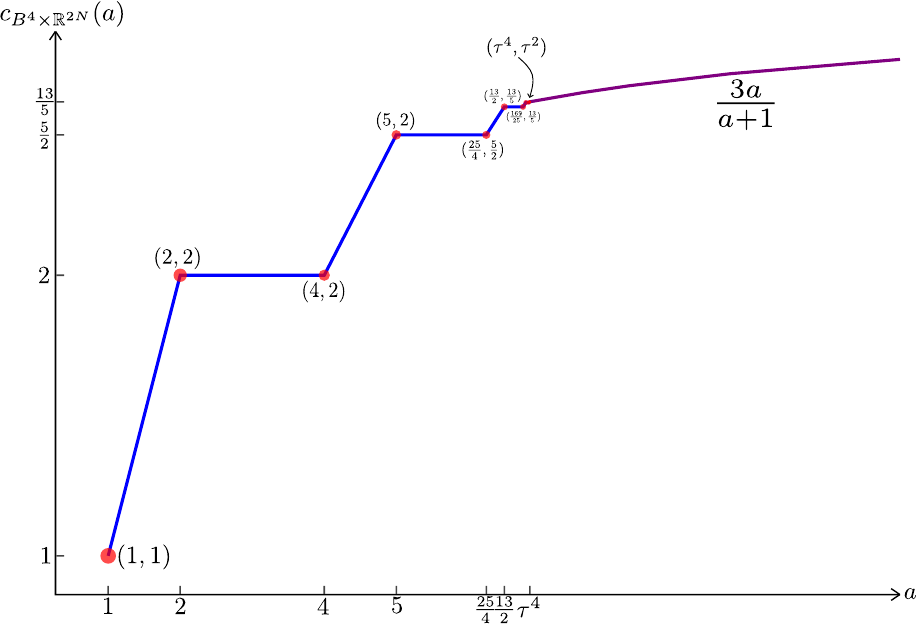}
  \caption{The stabilized ellipsoid embedding function $c_{B^4 \times \R^{2N}}(a)$ computed in Theorem~\ref{thmlet:main_SEEP} for $N \in \Z_{\geq 1}$. Note that for $a$ below the accumulation point $\tau^4$ this ``infinite staircase'' agrees with its unstabilized counterpart $c_{B^4}(a)$, while for $a > \tau^4$ the function $\tfrac{3a}{a+1}$ is a new purely high dimensional phenomenon.}
  \label{fig:stab_fib} 
  \end{figure}

Various special cases of Theorem~\ref{thmlet:main_SEEP} have been established previously, e.g. in \cite{HK,CGH,Ghost,Mint,chscI}. 
In particular, the embedding concocted by Hind in \cite{hind2015some} based on symplectic folding \cite{lalonde1995geometry,schlenk2003symplectic,Guth_polydisks} gives the upper bound $c_{B^4 \times \R^{2N}}(a) \leq \tfrac{3a}{a+1}$ for all $a \in \R_{\geq 1}$.  Hind also shows in \cite{hind2015some} that there is equality at all integers $a = 3k-1, k\ge 1$.
Obstructions giving matching lower bounds at the outer corner values $a = \be_0,\be_1,\be_2,\dots$ were proved using embedded contact homology in \cite{CGH}, and in fact these suffice to establish $c_{B^4 \times \R^{2N}}(a) = c_{B^4}(a)$ for all $a \in [1,\tau^4]$ by elementary scaling and monotonicity considerations. 
Thus the main new content of Theorem~\ref{thmlet:main_SEEP} is the lower bound $c_{B^4 \times \R^{2N}}(a) \geq \tfrac{3a}{a+1}$ for $a \in (\tau^4,\infty)$, though our approach also naturally recovers all previously known lower bounds as special cases. 

From the aforementioned works emerges a procedure for obstructing stabilized ellipsoid embeddings via moduli spaces of punctured pseudoholomorphic curves (\`a la symplectic field theory) with genus zero and one negative end, with prescribed asymptotic Reeb orbits. The main difficulty is to prove that the relevant moduli spaces are nonempty. 
For certain values of $a > \tau^4$, the necessary curves have been constructed via Hutchings--Taubes obstruction bundle gluing \cite{Mint} or neck-stretching closed rational curves with point constraints \cite{hind2015some, Ghost}, but attempts to push these methods further  yield diminishing returns.

Meanwhile, the papers \cite{chscI,SDEP} use algebraic structures arising in SFT to give recursive formulas (and even a closed tree formula in \cite{tree_formula}) which in principle can enumerate all of the relevant moduli spaces, where in particular a nonzero count implies nonemptiness. 
 With the aid of computer calculations these have been used to verify Theorem~\ref{thmlet:main_SEEP} in many additional cases, but proving general nonvanishing results by direct combinatorial analysis of these algorithms appears to be out of reach.\footnote{While some of these results rely on certain structural properties of symplectic field theory that are not yet fully established, the present article is entirely independent of these papers and does not depend on any virtual fundamental techniques in symplectic field theory. In the other direction, the results of this paper could be applied to compute the higher symplectic capacities $\{\gapac_\bb\}$ as in \cite{chscI} for the round four-ball and various other domains.}

Still more recently, in \cite{cusps_and_ellipsoids} we reformulated the above SFT moduli spaces in terms of closed rational pseudoholomorphic curves in $\CP^2$ with a distinguished $(p,q)$ cusp singularity, i.e. the singularity modeled on $\{x^p + y^q = 0\} \subset \C^2$.
 Since these curves could be a fortiori algebraic, this opens up the possibility of importing techniques from algebraic geometry in order to produce singular algebraic curves whose existence implies the relevant obstructions.\footnote{We will say that an algebraic curve $C$ in an algebraic surface $X$ has a $(p,q)$ cusp at a point $\pp \in C$ if there are open neighborhoods $\pp \in U \subset \CP^2$ and $(0,0) \in V \subset \C^2$ and a homeomorphism of pairs $(U,C \cap U) \cong (V,V \cap \{x^p + y^q = 0\})$.
Note that in this paper we are only considering singularities up to topological (as opposed to analytical) equivalence (see e.g. the notations and conventions section in \cite{greuel2018singular}).}
 Indeed, in \cite{mcduff2024singular} we observed that the obstructions at the outer corners $a = \be_0,\be_1,\be_2,\dots$ in Theorem~\ref{thmlet:main_SEEP} are carried by certain unicuspidal rational plane curves whose existence had been long known in certain circles (see \cite{orevkov2002rational,kashiwara_hiroko,fernandez2006classification}).

Under this reformulation, the curves relevant for $a \in (\tau^4,\infty)$ must have some additional singularities away from the distinguished cusp; such curves were called \hl{sesquicuspidal} in \cite{cusps_and_ellipsoids} because they generalize unicuspidal curves (i.e. those having one cusp and no other singularities).
We will deduce Theorem~\ref{thmlet:main_SEEP} from the following existence result for singular rational plane curves, whose formulation involves only classical algebraic geometry.

\begin{thmlet} \label{thmlet:sesqui_plane_curves}
Fix coprime integers $p > q > 1$ with $p+q$ divisible by $3$, and put $d := \tfrac{1}{3}(p+q)$.
There exists a rational algebraic curve in $\CP^2$ with a $(p,q)$ cusp and degree $d$ if and only if one of the following holds: 
  \begin{enumerate}[label=(\alph*)]
    \item\label{item:<tau^4} $(p,q) = (\fib_{k+4},\fib_{k})$ for some $k \in \Z_{\geq 3}$ odd
    \item\label{item:>tau_4} $p/q > \tau^4 := \tfrac{7+3\sqrt{5}}{2}$. 
  \end{enumerate} 
Moreover, these curves can be taken to be $(p,q)$-well-placed with respect to any given irreducible nodal cubic $\calN \subset \CP^2$.
\end{thmlet}

To explain the last sentence, note that such an $\calN \subset \CP^2$ is \hl{uninodal} (i.e. has one node and is otherwise nonsingular), say with local branches $\calB_-,\calB_+$ near its double point $\db \in \calN$. 
Following \cite[Def. 2.2.4]{mcduff2024singular}, we say that a curve $C$ is \hl{$(p,q)$-well-placed} with respect to $\calN$ if $C \cap \calN = \{\db\}$, $C$ is locally irreducible near $\db$, and we have local intersection numbers $(C \cdot \calB_-)_\db = p$ and $(C \cdot \calB_+)_\db = q$.
Since any two uninodal cubics in $\CP^2$ are projectively equivalent, for concreteness we often take $\calN = \calN_0 := \{x^3 + y^3 = xyz\}$.  
The well-placed condition is an auxiliary constraint on curves which looks unmotivated at first glance, but will be crucial
for establishing a connection with scattering diagrams in \S\ref{sec:from_wp_to_sd_and_back}, and also for exhibiting the symmetries discussed in \S\ref{sec:symmetries}.

\begin{rmk}\label{rmk:pq}
The condition $d = \tfrac{1}{3}(p+q)$ for a curve $C$ is equivalent to the \hl{index} $\ind_\C^{p,q}(C) := c_1(C) - p -q$ being zero, i.e. $C$ becomes rigid (at least virtually) after imposing a maximal order jet constraint at the cusp (see \cite[\S3]{cusps_and_ellipsoids} for details).  
\end{rmk}

\sss

We will see that in case \ref{item:<tau^4} of Theorem~\ref{thmlet:sesqui_plane_curves} we have $p/q < \tau^4$ and $\de_{d,p,q} = 0$, where $\de_{d,p,q} := \tfrac{1}{2}(d-1)(d-2) - \tfrac{1}{2}(p-1)(q-1)$ is the algebraic count of singularities away from the distinguished cusp, while in case \ref{item:>tau_4} we have $\de_{d,p,q} \geq 1$.
The curves in Theorem~\ref{thmlet:sesqui_plane_curves} are naturally organized by their value of $\de_{d,p,q} \in \Z_{\geq 0}$, which is preserved by certain symmetries $\Phi_{\CP^2},\Psi_{\CP^2}$ discussed in \S\ref{sec:symmetries}. In particular, putting $\de(C) := \de_{d,p,q}$ for a degree $d$ rational algebraic plane curve $C$ with a $(p,q)$ cusp, there are infinitely many other such curves (of arbitrarily high degree) with the same value of $\de(C)$.

\begin{rmk}
It follows a posteriori that the adjunction formula is a complete obstruction to the existence of the curves considered in Theorem~\ref{thmlet:sesqui_plane_curves}. In other words, for any $(p,q)$ and $d = \tfrac{1}{3}(p+q)$ not covered by \ref{item:<tau^4} or \ref{item:>tau_4}, we have $\de_{d,p,q} < 0$ (see \S\ref{sec:symmetries}).
\end{rmk}

An interesting feature of the curves in Theorem~\ref{thmlet:sesqui_plane_curves} is that they have very low degree relative to the cusp. 
Indeed, recall the following classical problem in algebraic curve theory (see e.g. \cite[Intro.]{greuel2021plane} or \cite[\S4.2.1(A)]{greuel2018singular}):

\begin{problem}\label{prob:min_deg}
  Determine the minimal degree $\degmin(p,q)$ of any\footnote{We could alternatively restrict to rational curves here. Since our constructions give rational curves and the adjunction obstruction holds a fortiori for higher genus curves, Corollary~\ref{cor:most_d_minimal} below holds equally in this case.} algebraic curve in $\CP^2$ which has a $(p,q)$ singularity (and possibly other singularities).
\end{problem}

\NI A state-of-the-art result can be found in \cite[Thm. 3.10]{greuel2021plane}, which implies $\degmin(p,q) \leq 3\sqrt{(p-1)(q-1)} - 1$ for all coprime $p,q \in \Z_{\geq 1}$.
We refer the reader to \cite[\S4.5.5]{greuel2018singular} or \cite[\S3.3]{greuel2021plane} for historical context and more general results.
By combining Theorem~\ref{thmlet:sesqui_plane_curves} with adjunction considerations, we resolve Problem~\ref{prob:min_deg} for ``most'' values of $(p,q)$ with $p+q$ divisible by $3$.
As a shorthand, when $p+q$ is divisible by $3$ we will put $\de_{p,q} := \de_{d,p,q}$ with $d := \tfrac{1}{3}(p+q)$.

\begin{corlet}\label{cor:most_d_minimal}

Let $C$ be one of the rational algebraic plane curves provided by Theorem~\ref{thmlet:sesqui_plane_curves}, say with a $(p,q)$ cusp and degree $d = \tfrac{1}{3}(p+q)$.
Then we have 
\begin{align}\label{eq:degmin_as_expected}
\degmin(p,q) = \tfrac{1}{3}(p+q)
\end{align}
unless $d \leq 2 + \de_{p,q}$.
In particular, \eqref{eq:degmin_as_expected} holds if $\de_{p,q} \leq 4$, and in general it holds for all but finitely many
coprime $p,q \in \Z_{\geq 1}$ with  $p+q$ divisible by $3$ and fixed value of $\de_{p,q}$.
\end{corlet}

\begin{proof}
In order to prove minimality of $d$, it suffices to show that any algebraic curve of degree $d-1$ with a $(p,q)$ cusp is ruled out by the adjunction formula, i.e. that we have
\begin{align*}
\de^-_{d,p,q} := \tfrac{1}{2}(d-2)(d-3) - \tfrac{1}{2}(p-1)(q-1) < 0.
\end{align*}
In other words, if $d$ is not minimal then we have $\de^-_{d,p,q} \geq 0$, and hence $d-2 = \de_{d,p,q} - \de^-_{d,p,q} \leq \de_{d,p,q}$, i.e. $d \leq 2 + \de_{d,p,q}$. We then have
\begin{align*}
\de_{d,p,q} \leq \tfrac{1}{2}(d-1)(d-2) - \tfrac{1}{2}(3d-3) = \tfrac{1}{2}(d-1)(d-5) \leq \tfrac{1}{2}(\de_{d,p,q}+1)(\de_{d,p,q}-3),
\end{align*}
which implies $\de_{d,p,q} \geq 5$. Finally, the last claim in the corollary follows since we have an a priori upper bound on $p+q$ when \eqref{eq:degmin_as_expected} fails.
\end{proof}

\begin{rmk}
Theorem~\ref{thmlet:sesqui_plane_curves} also implies a new (a priori weaker) result on the existence of $(p,q)$-sesquicuspidal rational symplectic curves in $\CP^2$, i.e. working in the symplectic rather than algebraic category, and this already suffices to prove Theorem~\ref{thmlet:main_SEEP} (see \S\ref{sec:emb_obs}).
In particular, in the language of \cite{etnyre2020symplectic}, we get new genus zero projective symplectic hats (often of minimal degree) of the transverse torus knot $\TT(p,q)$ with maximal self-linking number, and by loc. cit. these are equivalent to positively immersed symplectic cobordisms from $\TT(p,q)$ to the standard transverse torus link $\TT(d,d)$ with $d = \tfrac{1}{3}(p+q)$.
It appears quite challenging to construct these cobordisms using more flexible topological techniques, for instance by manipulating braid diagrams (c.f. \cite{feller2021genus} or \cite[\S6.1]{chaidez2021lattice}). 
\end{rmk}

Let us also mention that the result in Theorem~\ref{thmlet:main_SEEP} is robust under certain perturbations of the round ball.
Consider the trapezoid in $\R^2_{\geq 0}$ with vertices $(0,0),(0,2),(2,1),(2,0)$, and let $\trapdom \subset \C^2$ denote its preimage under the moment map $\C^2 \ra \R_{\geq 0}^2, (z_1,z_2) \mapsto (\pi |z_1|^2,\pi |z_2|^2)$ for the standard Hamiltonian torus action on $\C^2$.

\begin{corlet}
Let $U \subset \C^2$ be any open subset
 such that $\trapdom \subset U \subset B^4(3) := E(3,3)$.
 Then we have
 $c_{U \times \R^{2N}}(a) = \tfrac{a}{a+1}$ for any $a \geq \tau^4$ and $N \geq 1$.
\end{corlet}
\begin{proof}
  Theorem~\ref{thmlet:main_SEEP} together with monotonicity under symplectic embeddings and scaling considerations gives $c_{U \times \R^{2N}}(a) \leq c_{B^4(3) \times \R^{2N}}(a) = \tfrac{a}{a+1}$ for all $a \geq \tau^4$ and $N \geq 1$. On the other hand, by \cite[Prop. 3.1]{cristofaro2022higher} we have the folding-type symplectic embedding $E(\tfrac{1}{\mu},\tfrac{a}{\mu}) \times \R^{2N} \hooksymp \trapdom \times \R^{2N}$ for all $\mu > \tfrac{a}{a+1}$, which
gives $c_{U \times \R^{2N}}(a) \geq c_{\trapdom \times \R^{2N}}(a) \geq \tfrac{a}{a+1}$ for all $a \geq 1$ and $N \geq 1$.
\end{proof}

\sss

A simple byproduct of our proof of Theorem~\ref{thmlet:main_SEEP} is that we have $c_{B^4 \times \R^{2N}}(a) = c_{\CP^2 \times \R^{2N}}(a)$ for all $a \in \R_{\geq 1}$ and $N \in \Z_{\geq 1}$,\footnote{Here $\CP^2 := \CP^2(1)$ is equipped with the Fubini--Study form normalized so that a line has area $1$. It should generally be clear from the context whether we are viewing a given space as a complex algebraic surface or a symplectic four-manifold.} and it is natural to try to replace $\CP^2$ with other del Pezzo surfaces.
Recall that by definition these are smooth Fano complex projective surfaces, which up to diffeomorphism are $\CP^1 \times \CP^1$ and $\bl^j\CP^2 := \CP^2 \#^{\times j}\ovl{\CP}^2$ for $j = 0,\dots,8$. 
Up to symplectomorphism each of these admits a unique symplectic form which is \hl{unimonotone} (i.e. the first Chern class and symplectic area class coincide -- see e.g. \cite{salamon2013uniqueness}), and the complex structure is rigid when the degree is at least $5$ (i.e. $\CP^1 \times \CP^1$ and $\bl^j\CP^2$ for $j=0,\dots,4$).
For example, the Fubini--Study symplectic form on the complex projective plane becomes unimonotone after rescaling so that a line has symplectic area $3$ (this is sometimes denoted by $\CP^2(3)$).

It was recently observed in \cite{cristofaro2020infinite,casals2022full} that, for each unimonotone rigid\footnote{By slight abuse, rigidity here refers to the complex structure, even though at the moment we are viewing $X$ as a symplectic manifold.} del Pezzo surface $X$, the corresponding four-dimensional ellipsoid embedding function $c_X(a)$ is an infinite staircase in the regime $1 \leq a \leq a_\acc^X$, analogous to the one appearing in Figure~\ref{fig:stab_fib}, with numerics given by solutions to a recursive equation $g_{k+2J} = K g_{k+J} - g_k$. Here $K+2$ is the degree of the del Pezzo surface $X$, $J$ is called the number of 
\hl{strands} of the staircase, and $a_\acc^X = \tfrac{1}{2}(K + \sqrt{K^2-4})$ is the \hl{accumulation point}
 (see e.g. \cite[\S2.4]{cusps_and_ellipsoids} for a more detailed overview). 
For example, for $X = \CP^2$ we have $K=7$, $J=2$, and $a_\acc^{\CP^2} = \tau^4$. Note that $J$ is not directly visible from Figure~\ref{fig:stab_fib}, but it is the number of initial ``seeds'' needed to generate all of the steps via the above recursion, 
with $J=2$ for $\CP^1,\CP^1\times \CP^1,\bl^3\CP^2,\bl^4\CP^2$
and $J=3$ for $\bl^3\CP^2,\bl^4\CP^2$. 
According to \cite[Cor. C]{mcduff2024singular}, these infinite staircases are stable, i.e. for each unimonotone rigid del Pezzo surface we have $c_X(a) = c_{X \times \R^{2N}}(a)$ for all $a \in [1,a_\acc^X]$ and any $N \in \Z_{\geq 1}$.
The following theorem addresses what happens beyond the accumulation point.

\begin{thmlet}\label{thmlet:SEEP_dPs} 
Fix $N \in \Z_{\geq 1}$.
\begin{enumerate}[label=(\alph*)]
\item\label{item:SEEP_rigid} If $X$ is one of the unimonotone rigid del Pezzo surfaces $\bl^j\CP^2$ 
for $j \in \{0,1,2,3\}$ or $\CP^1 \times \CP^1$, we have
\begin{align}\label{eq:rigid_dP_SEEP}
c_{X \times \R^{2N}}(a) = 
\begin{cases}
  c_X(a) &  \text{if}\;\; a \in [1,a_\acc^X]\\
  \tfrac{a}{a+1} & \text{if}\;\;   a \in [a_\acc^{X},\infty).
\end{cases}
\end{align}

\item\label{item:SEEP_nonrigid} For the unimonotone del Pezzo surfaces $\bl^j\CP^2$
for $j \in \{5,6,7,8\}$, we have
\begin{align}\label{eq:higher_dP_SEEP}
c_{X \times \R^{2N}}(a) \geq \tfrac{a}{a+1} \;\;\;\;\;\text{for all} \;\;\;\;\; a \in [1,\infty).
\end{align}
\end{enumerate}
\end{thmlet}
\begin{rmk}\label{rmk:SEEP_for_4_pt_bu}
For the missing unimonotone del Pezzo surface $X = \bl^4\CP^2$, 
 $c_{X \times \R^{2N}}(a)$ is still bounded from below by the right hand side of \eqref{eq:rigid_dP_SEEP} for all $a \in [1,\infty)$, and  equal to it for $a \in [1,a_\acc^X]$.
Proving the matching upper bound $c_{X \times \R^{2N}}(a) \leq \tfrac{a}{a+1}$ for $X = \bl^j\CP^2$
with $j \geq 4$ appears to require a refinement of the explicit folding-type symplectic embedding in \cite[Prop. 3.1]{cristofaro2022higher} (see also Remark~\ref{rmk:eliashberg_princ}).
\end{rmk}

Just as in the case of Theorem~\ref{thmlet:main_SEEP}, the proof of Theorem~\ref{thmlet:SEEP_dPs}
is based on  the following result about the  existence of enough singular rational algebraic curves, this time  in general del Pezzo surfaces. We already know from
 \cite[Thm. B]{mcduff2024singular} that for
  each outer corner of the infinite staircase $c_X|_{[1,a_\acc^X]}$, say with $x$-value $p/q$, there exists a $(p,q)$-unicuspidal rational algebraic curve $C$ in $X$ with $p+q = c_1([C])$. For brevity we will sometimes refer to these as the \hl{outer corner curves} of $X$.

\begin{thmlet}\label{thmlet:curves_in_X} Let $X$ be a del Pezzo surface, and let 
$\Sset_X$ denote the set of reduced fractions $p/q \geq 1$ such that there exists a rational algebraic curve $C$ in $X$ with a $(p,q)$ cusp satisfying $p+q = c_1([C])$.

\begin{enumerate}[label=(\alph*)]

  \item\label{item:curves_in_X_rigid} If $X$ is rigid (i.e. has degree $\geq 5$), then $\Sset_X$ is dense in $[a_\acc^X,\infty)$. Meanwhile, $\Sset_X\cap [1,a_\acc^X)$ is a discrete set accumulating at $a_\acc^X$, consisting precisely of the $x$-values of outer corners of the infinite staircase $c_X|_{[1,a_\acc^X]}$.

  \item\label{item:curves_in_X_nonrigid} If $X$ is nonrigid (i.e. has degree $\leq 4$), then $\Sset_X$ is dense in $[1,\infty)$.
\end{enumerate}
Moreover, the relevant curves can be taken to be well-placed with respect to any given uninodal anticanonical divisor.
\end{thmlet}

\begin{rmk}\label{rmk:jgeq9case} \hfill
\begin{enumerate}
  \item 
It also follows from our proof of Theorem~\ref{thmlet:curves_in_X} that $\Sset_X$ is dense in $[1,\infty)$ if $X$ is any blowup of $\CP^2$ at $j \geq 9$ points in suitably general position, or if all of the points lie in the smooth locus of a given uninodal cubic (see Remark~\ref{rmk:gen_rat_surf_II}). 
We expect that such curves in non-Fano rational surfaces will play an important role in the general case of Problem~\ref{prob:SEEP}.

\item Our proof actually gives a precise description of $\Sset_X$ in many cases (c.f. Remark~\ref{rmk:precise_density}). 
For instance, in the case $X = \CP^1 \times \CP^1$, a reduced fraction $p/q > a_\acc^X$ lies in $\Sset_X$ if and only if $p+q$ is divisible by $2$, while for $X = \bl^j\CP^2$ with $j \in \{3,4\}$ we have simply $\Sset_X \cap [a_\acc^X,\infty) = \Q \cap [a_\acc^X,\infty)$. 
\end{enumerate}
\end{rmk}

\sss

A key idea underlying Theorems ~\ref{thmlet:sesqui_plane_curves} and \ref{thmlet:curves_in_X} is an explicit connection between well-placed curves and scattering diagrams. 
As we recall in \S\ref{subsec:scattering_review}, a \hl{scattering diagram} $\calD$ is an algebro-combinatorial object which consists roughly of a collection of oriented rays in $\R^2$, each labeled by a power series in $\C[x,x^{-1},y,y^{-1}]\ll t \rr$.
Scattering diagrams were defined implicitly by Kontsevich--Soibelman \cite{Kontsevich2006} as a bookkeeping tool for Maslov index zero holomorphic disks in singular Lagrangian toric fibrations, and were subsequently studied extensively (also in higher dimensions) as part of the Gross--Siebert approach to mirror symmetry (see e.g. \cite{gross2011real,gross2015mirror,GHS_theta,gross2022canonical}). They also arise naturally in various other contexts, for instance cluster algebras \cite{GHKK2018,gross2010quivers,cheung2017greedy,scattering_fans}, and quiver representations \cite{reineke2010poisson,reineke2011cohomology,bridgeland2017,reineke2013refined}.
A scattering diagram $\calD$ carries a natural notion of monodromy around closed loops in $\R^2$, and the Kontsevich--Soibelman algorithm produces a new scattering diagram $\ks(\calD)_\min$ by adding labeled rays (typically infinitely many) in order to kill the monodromy. 
This process introduces rich combinatorics which are not fully understood even when the initial scattering diagram $\calD$ is very simple.

A natural setting for our correspondence result is that of \hl{uninodal Looijenga pairs} $(X,\calN)$, i.e. $X$ is a complex projective surface and $\calN \subset X$ is a rational anticanonical divisor with one double point.
Following \cite{gross2015mirror}, any such pair admits at least one \hl{toric model}, which in particular identifies the ``interior'' $X \setminus \calN$ with the interior of a nontoric blowup of a toric surface (see \S\ref{subsec:tor_mod_prelim} for details).
By way of notation, put $z^{\mm} := x^ay^b \in \C[x,x^{-1},y,y^{-1}]$ for each $\mm = (a,b) \in \Z^2$. 
\begin{thmlet}[Theorem~\ref{thm:wp_and_sd} and \S\ref{subsec:rigid_dP_tor_mods}]\label{thmlet:wp_and_sd}
 For each uninodal Looijenga pair $(X,\calN)$ with strongly convex\footnote{Here strong convexity is a technical condition for toric models which holds in all of our main examples -- see \S\ref{subsec:wp_and_GW}.} toric model $\tormod$, there is a scattering diagram $\calD_{\tormod}$ and a piecewise linear map $\wp: \Z^2_{\geq 1} \ra \Z^2$ such that, for each coprime $p,q \in \Z_{\geq 1}$, the following are equivalent:
\begin{enumerate}[label=(\alph*)]
  \item there exists a rational algebraic curve in $X$ which is $(p,q)$-well-placed with respect to $\calN$
  \item the coefficient of $z^{\wp(p,q)}$ in $\ff_{\wp(p,q)}$ is nonzero as an element of $\C\ll t \rr$, where $\ff_{\wp(p,q)} \in \C[x,x^{-1},y,y^{-1}]\ll t \rr$ is the label attached to the ray $\R_{\geq 0} \cdot \wp(p,q)$ in $\ks(\calD_\tormod)_\min$.
\end{enumerate}

Moreover, in the case that $X$ is a rigid del Pezzo surface, there is a strongly convex toric model $\tormod_X$ such that $\calD_{\tormod_X}$ has rays $\R_{\leq 0} \cdot \mm_i$ labeled by functions $\ff_i = (1+tz^{\mm_i})^{\ell_i}$ for $i=1,\dots,J$ as in Table~\ref{table:fund_bij}.
\end{thmlet}

As we point out in Remark~\ref{rmk:atfs_and_tms}, toric models for uninodal Looijenga pairs are closely related to four-dimensional symplectic almost toric fibrations (as in e.g. \cite{symington71four}).
In particular, for those unimonotone rigid del Pezzo surfaces $X$ admitting triangular almost toric fibrations (namely 
$\bl^j\CP^2$
for $j=0,3,4$ and $\CP^1 \times \CP^1$, corresponding to $J=2$ strands), the initial scattering diagram $\calD_{\tormod_X}$ has just two rays, and hence can be converted (by a change of lattice trick as in \cite[\S C.3]{GHKK2018}) to one of the basic scattering diagrams $\calD_{e_1,e_2}^{\ell_1,\ell_2}$ studied in \cite{gross2010quivers} (see \S\ref{sec:std_sd}). For instance, in the case $X=\CP^2$, Theorem~\ref{thmlet:wp_and_sd} reduces the existence of well-placed rational plane curves to understanding
the scattering diagram $\ks(\calD_{e_1,e_2}^{3,3})_\min$. 
These basic scattering diagrams were first studied using computer experiments by Kontsevich, and were discussed empirically in e.g. \cite[Ex. 1.6]{gross2010tropical} and \cite[Ex. 1.15]{GHKK2018}. 
In the case $\ell_1 = \ell_2$, Gross--Pandharipande \cite{gross2010quivers} exploited a surprising connection with quiver representation theory due to Reineke \cite{reineke2010poisson} in order to completely describe all rays appearing in $\ks(\calD_{e_1,e_2}^{\ell_1,\ell_1})_\min$. Together with Theorem~\ref{thmlet:wp_and_sd}, this suffices to prove Theorem~\ref{thmlet:sesqui_plane_curves}, and, combined with additional symmetry considerations and a blowup trick in \S\ref{sec:symmetries}, also Theorem~\ref{thmlet:curves_in_X} except for the case $X = \CP^1 \times \CP^1$.
More recently, Gross--Hacking--Keel--Kontsevich \cite{GHKK2018} proved a powerful positive factorization result for scattering diagrams, and Gr\"afnitz--Luo \cite{grafnitz2023scattering} combined this with the deformation techniques from \cite{gross2010tropical} and the symmetries $T_1,T_2$ from \cite{gross2010quivers} in order to combinatorially analyze the scattering diagrams $\ks(\calD_{e_1,e_2}^{\ell_1,\ell_1})_\min$ for all $\ell_1,\ell_2 \in \Z_{\geq 1}$. 
In particular, the resulting scattering positivity results are strong enough to handle the remaining case of $\CP^1 \times \CP^1$, which corresponds to $(\ell_1,\ell_2) = (2,4)$.

\begin{rmk}
It is remarkable that the shape of the stabilized ellipsoid embedding function $c_{B^4 \times \R^{2N}}(a)$ depicted in Figure~\ref{fig:stab_fib} precisely mirrors the rays in the scattering diagram $\ks(\calD_{e_1,e_2}^{3,3})_\min$, especicially since the former is characterized in the regime $[\tau^4,\infty)$ by Hind's symplectic folding construction \cite{hind2015some}, which appears very flexible and nonalgebraic in nature.
\end{rmk}

\begin{rmk}
In this paper the flow of implications go mostly from scattering diagrams to quantitative symplectic geoemtry (with singular algebraic curves as an intermediary), but we also expect that this conceptual framework will be useful in proving new nontrivial resuls about scattering diagrams -- see e.g. Remark~\ref{rmk:refined_SD}. 
\end{rmk}

\section*{Organization of the paper}
We explain in \S\ref{sec:emb_obs} how to deduce Theorem E (and hence also Theorem A) from Theorem F.  In \S\ref{sec:symmetries} we use symmetries to show that Theorem B  implies all cases of  Theorem F except for $X: = \C P^1\times \C P^1$.  \S\ref{sec:fund_bij} discusses properties of Looijenga pairs and constructs explicit 
toric models for each rigid $X$, and establishes in Proposition~\ref{prop:fund_bij} a bijection between $(p,q)$-well-placed rational curves in $X$ and algebraic curves in the corresponding toric model.
\S\ref{sec:from_wp_to_sd_and_back} explains how one can use scattering diagrams to prove the existence of well-placed algebraic curves, and finally \S\ref{sec:std_sd} explains how known results about scattering diagrams can be used to establish the existence of the curves needed to prove Theorem B, and also Theorem F in the case $X: = \C P^1\times \C P^1$.

\section*{Acknowledgements} K.S. would like to thank Mark Gross and Bernd Siebert for explaining the change of lattice trick from \cite[\S C.3]{GHKK2018}, and also for pointing out several other relevant references related to scattering diagrams. We also thank J\'anos Koll\'ar for various interesting comments and questions.

\section{Obstructing symplectic embeddings via algebraic curves}\label{sec:emb_obs}

In this section we briefly explain how to deduce Theorem~\ref{thmlet:SEEP_dPs} (and hence also Theorem~\ref{thmlet:main_SEEP} as a special case) from Theorem~\ref{thmlet:curves_in_X}.
Theorem~\ref{thm:stab_obs_from_curve} below also implies various refinements of Theorem~\ref{thmlet:SEEP_dPs}, wherein the domain $E(a,b) \times \R^{2N}$  is replaced with a compact symplectic manifold; we leave the explicit formulations to the interested reader. 

An algebraic curve $C$ in a smooth complex projective surface is \hl{weakly $(p,q)$-sesquicuspidal} (for some coprime $p,q \in \Z_{\geq 1}$) if it has a $(p,q)$ cusp singularity, and \hl{$(p,q)$-sesquicuspidal} if all of its auxiliary singularities are ordinary double points. We define the (complex) \hl{index} of such a curve to be $\ind_\C^{p,q}(C) := c_1([C]) - p - q$. 
Symplectic $(p,q)$-sesquicuspidal curves are defined similarly but require $C$ to be only a symplectic submanifold away from the singular points (note that the auxiliary double points are required to be positively oriented).
In the symplectic category we can always perturb the auxiliary singularities into finitely many ordinary double points, although this is not guaranteed in the algebraic category.

The following is our main tool for deducing symplectic embedding obstructions from singular curves.

\begin{thm}[{Cor. 2.7.2, Cor. 2.3.8, Thm. D, and Thm. E in \cite{cusps_and_ellipsoids}}]\label{thm:stab_obs_from_curve}
Let $X$ be a closed symplectic four-manifold, and suppose that $X$ contains an index zero\footnote
{
See Remark~\ref{rmk:pq}.} 
 $(p,q)$-sesquicuspidal rational symplectic curve $C$ in homology class $A \in H_2(X)$ for some coprime $p,q \in \Z_{\geq 1}$.
Then any symplectic embedding $E(\tfrac{1}{\mu} \cdot (p,q,b_1,\dots,b_N)) \hooksymp X \times Q^{2N}$ with $Q^{2N}$ a closed symplectic manifold of dimension $2N \geq 0$ and $b_1,\dots,b_N > pq$ must satisfy $\mu \geq \tfrac{pq}{\area(A)}$, provided that $X \times Q$ is semipositive.\footnote{\hl{Semipositivity} is a technical condition which is automatic if $X \times Q$ has real dimension at most $6$, or if $X \times Q$ is monotone (i.e. the first Chern class and symplectic area class are positively proportional). This assumption could probably be removed by more abstract perturbation techniques.} 
\end{thm}

We will mostly take Theorem~\ref{thm:stab_obs_from_curve} as a black box, but the rough idea is as follows.
Firstly, we can find a compatible almost complex structure $J$ on $X$ which preserves $C$, and we consider the moduli space of $J$-holomorphic curves of the same type as $C$, after imposing an additional jet constraint to cut down the expected dimension to zero. By a version of automatic transversality in dimension four, every such curve counts positively, and the same is true if we view these curves as lying in a slice $X \times \{\op{pt}\}$ of $X \times Q$, where the latter is equipped with a split almost complex structure.
By a compactness argument, the corresponding moduli space of curves in $X \times Q$ persists under general deformations of the almost complex structure.
Moreover, we can trade the cusp singularity for an asymptotic negative end in (the symplectic completion of) $X \times Q$ minus a suitable ellipsoid.
Finally, given a hypothetical symplectic embedding $E(\tfrac{1}{\mu} \cdot (q,p,b_1,\dots,b_N)) \hooksymp X \times Q^{2N}$, the desired inequality follows by 
 noting that the  (appropriately normalized) 
symplectic area of every curve in this (necessarily nonempty) moduli space must be positive.

Note that for any symplectic embedding $E(a_1,a_2) \times \R^{2N} \hooksymp X \times \R^{2N}$ and 
any $b_1,\dots,b_N$, the image of $E(a_1,a_2,b_1,\dots,b_N)$ lands in $X \times B^{2N}(b)$ for some finite $b$, and hence there is also a symplectic embedding $E(a_1,a_2,b_1,\dots,b_N) \hooksymp X \times Q^{2N}$ for some closed symplectic manifold $Q$ of dimension $2N$. Thus Theorem~\ref{thm:stab_obs_from_curve} implies:

\begin{cor}\label{cor:stab_obs_from_curve}
  Let $X$ be a closed symplectic four-manifold which contains an index zero $(p,q)$-sesquicuspidal rational symplectic curve in homology class $A \in H_2(X)$ for some coprime $p,q \in \Z_{\geq 1}$. Then for all $N \in \Z_{\geq 0}$ we have 
  $c_{X \times \R^{2N}}(p/q) \geq \tfrac{p}{\area(A)}$, provided that either $X$ is monotone or $N \leq 1$. 
\end{cor}

Using Corollary~\ref{cor:stab_obs_from_curve}, we now deduce Theorem~\ref{thmlet:SEEP_dPs} from Theorem~\ref{thmlet:curves_in_X}.

\begin{proof}[Proof of Theorem~\ref{thmlet:SEEP_dPs}]
  
Given a smooth complex projective surface $X$ and a rational algebraic curve $C$ in $X$ with a $(p,q)$ cusp satisfying $p+q = c_1([C])$, a small perturbation of $C$ gives an index zero $(p,q)$-sesquicuspidal rational symplectic curve in $X$. If $X$ is unimonotone, then Corollary~\ref{cor:stab_obs_from_curve} gives the embedding obstruction
\begin{align}
c_{X \times \R^{2N}}(p/q) \geq \tfrac{p}{\area([C])} = \tfrac{p}{c_1([C])} = \tfrac{p}{p+q} = \tfrac{(p/q)}{(p/q)+1}. \label{eq:lb_from_curve} 
\end{align}

In \ref{item:SEEP_rigid}, the upper bound $c_{X \times \R^{2N}}(a) \leq c_X(a)$ for all $a \geq 1$ is manifest by stabilizing four-dimensional symplectic embeddings.
In the other direction, by combining the outer corner curves in Theorem~\ref{thmlet:curves_in_X}\ref{item:curves_in_X_rigid} with \eqref{eq:lb_from_curve}, 
we get the lower bound $c_{X \times \R^{2N}}(a) \geq c_X(a)$ whenever $a$ is the $x$-value of an outer corner of the infinite staircase in $c_{X}|_{[1,a_\acc^X]}$, and hence for all $1 \leq a \leq a_\acc^X$ by scaling and monotonicity considerations (see \cite[Cor. C]{mcduff2024singular} for more details).
Similarly, the curves with $p/q > a_\acc^X$ in Theorem~\ref{thmlet:curves_in_X}\ref{item:curves_in_X_rigid} give the lower bound $c_{X \times \R^{2N}}(a) \geq \tfrac{a}{a+1}$ for all $a$ in a dense subset of $[a_\acc^X,\infty)$, and hence for all $a$ in $[a_\acc^X,\infty)$ by continuity of $c_{X \times \R^{2N}}$. 
Meanwhile, the upper bound $c_{X \times \R^{2N}}(a) \leq c_X(a)$ for all $a > a_\acc^X$ follows from the explicit symplectic embedding constructed in \cite[Prop. 3.1]{cristofaro2022higher} (see \cite[Prop. 7.2.1]{mcduff2024singular}). 
As for \ref{item:SEEP_nonrigid}, the lower bound $c_{X \times \R^{2N}}(a) \geq \tfrac{a}{a+1}$ for all $a \geq 1$ follows similarly by combining Theorem~\ref{thmlet:curves_in_X}\ref{item:curves_in_X_nonrigid} with \eqref{eq:lb_from_curve}.
\end{proof}

\begin{rmk}
For coprime $p,q \in \Z_{\geq 1}$ with $p+q$ divisible by $3$, the above proof actually gives $c_{\CP^2 \times \R^{2N}}(1,p/q,b_1,\dots,b_N) = c_{\CP^2 \times \R^{2N}}(p/q)$ whenever $b_1,\dots,b_N > p$, where we put
$c_{X^4 \times \R^{2N}}(\veca) := \inf\left\{\mu\;|\; E(\tfrac{1}{\mu} \cdot \veca) \hooksymp X^4 \times \R^{2N}\right\}$ for $\veca \in \R^{2+N}_{>0}$.
As a step towards computing $c_{\CP^2 \times \R^{2N}}(\veca)$, it is natural to ask what is the supremal $b$ for which $c_{\CP^2 \times \R^{2N}}(1,p/q,b,\dots,b) \neq c_{\CP^2 \times \R^{2N}}(p/q)$.
\end{rmk}  

\sss

We end this section with a heuristic which suggests that the lower bounds in Theorem~\ref{thmlet:SEEP_dPs}\ref{item:SEEP_nonrigid} (and also Remark~\ref{rmk:SEEP_for_4_pt_bu}) should be optimal.
\begin{rmk}\label{rmk:eliashberg_princ}
For a unimonotone closed symplectic four-manifold $X$, we observed above that any index zero $(p,q)$-sesquicuspidal symplectic rational curve in $X$ gives the lower bound 
\begin{align}\label{eq:HK_lb}
c_{X \times \R^{2N}}(a) \geq \tfrac{a}{a+1}
\end{align}
 for $a = p/q$. 
By trading the cusp for a negative puncture as in \cite[Thm. D]{cusps_and_ellipsoids},
this lower bound can be understood in terms of index zero rational pseudoholomorphic planes in the completed symplectic cobordism $\wh{W}$, 
where $W$ denotes $X \times Q$ minus a small ellipsoid with area factors proportional to $(p,q,b_1,\dots,b_N)$, with $b_1,\dots,b_N \gg p,q$ and $Q^{2N}$ a sufficiently large closed symplectic manifold.
Here the fact that we are considering curves with genus zero and exactly one negative end translates crucially into the fact that the index does not change as we stabilize from $X$ to $X \times Q$ (see e.g. \cite[\S2.6]{cusps_and_ellipsoids}).

In general, one could imagine using various other moduli spaces of asymptotically cylindrical punctured curves in $\wh{W}$ (\`a la symplectic field theory) in order to read off lower bounds for $c_{X \times \R^{2N}}$ which might improve upon \eqref{eq:HK_lb}. 
These moduli spaces should have nonnegative index in order to have a reasonable chance at persisting under deformations, but in principle they could include curves of higher genus, with multiple negative ends, and possibly carrying additional constraints.
However, by purely formal considerations based on index and action, one can show that these  lower bounds (which are given ultimately by 
using the positivity of symplectic area
as above) cannot improve upon \eqref{eq:HK_lb}; see e.g. \cite[\S7]{siegel2023symplectic} for an analogous formal argument in the setting of higher dimensional ball packings.
Thus, Eliashberg's ``holomorphic curves or nothing'' metaprinciple suggests that the lower bound $c_{X \times \R^{2N}}(a) \geq \tfrac{a}{a+1}$ should be optimal whenever it holds.
\end{rmk}

\section{Symmetries of del Pezzo surfaces}\label{sec:symmetries}

Our main goal in this section is to deduce Theorem~\ref{thmlet:curves_in_X} (except for the case $X = \CP^1 \times \CP^1$) from Theorem~\ref{thmlet:sesqui_plane_curves}.
The basic idea is to blow up (weakly) sesquicuspidal curves in the projective plane in order to embed them into lower degree del Pezzo surfaces, and then to apply the birational symmetries $\Phi_X,\Psi_X$ from \cite[\S2.3]{mcduff2024singular} in order to further enlarge these curve families.
The case of  $\CP^1 \times \CP^1$ is exceptional and is postponed until \S\ref{sec:std_sd}.
Incidentally, the scattering diagram argument given in \S\ref{subsec:wp_and_basic_sd} will independently prove a more explicit version of the density statement in Theorem~\ref{thmlet:curves_in_X}(a) in the four cases with $J=2$ strands (including $\CP^1 \times \CP^1$), whereas the argument in this section applies uniformly for blowups of complex projective space irrespective of $J$.

Let $\calN$ be a uninodal anticanonical divisor in a del Pezzo surface $X$, and let $\calB_-,\calB_+$ denote the two local branches of $\calN$ near its double point $\db$.
Let $K+2$ denote the degree of $X$, which is by definition $\calN \cdot \calN \in \{1,\dots,9\}$.
We will say that a curve $C$ in $X$ is \hl{$(p,q)$-well-placed with respect to $(\calN;\calB_-,\calB_+)$} if $C$ intersects $\calN$ only at $\db$, $C$ is locally irreducible near $\db$, and we have local intersection multiplicities $(C \cdot \calB_-)_{\db} = p$ and $(C \cdot \calB_+)_\db = q$.
Note that this implies $p+q = c_1([C])$, i.e. $\ind_\C^{p,q}(C) = 0$, and, for $p,q$ coprime, $C$ has a $(p,q)$ cusp.
We will often suppress $\calB_-,\calB_+$ (and sometimes also $\calN$) from the notation if the local branches (and the divisor itself) are implicit or immaterial.
\begin{convention}\label{conv:pq=0}
It will also be convenient to slightly extend the above definition by saying that $C$ is \hl{$(k,0)$-well-placed} if $C \cap \calN$ is a single smooth point of $\calN$ and $C \cdot \calN = k$, with $(0,k)$-well-placed curves defined in the exact same way.
\end{convention}

For $K+2 \geq 5$, \cite[\S2]{mcduff2024singular} defines a birational transformation $\Phi_X: X \dashrightarrow X$ which proceeds by blowing up $X$ at $K$ points infinitely near to the node of $\calN$, each corresponding to the branch $\calB_+$, and then blowing down $K$ times starting with the strict transform of $\calN$, and finally identifying the result isomorphically with $X$. 
Swapping the roles of $\calB_+$ and $\calB_-$ gives another birational transformation $\Psi_X: X \dashrightarrow X$.
Koll\'ar \cite{Kol} extends these birational transformations to the case $K+2 = 4$, and also upgrades them to isomorphisms of $X \setminus \calN$.\footnote{Note that $\Phi_X,\Psi_X$ are denoted by $\si_+,\si_-$ respectively in \cite{Kol} and refered to as \hl{Geiser-type involutions}.}
These papers show that if $C$ is $(p,q)$-well-placed with respect to $(\calN;\calB_-,\calB_+)$, then $\Phi_X(C)$ is $(p',q')$-well-placed with respect to $(\calN;\calB_-,\calB_+)$, where 
\begin{itemize}
  \item $(p',q') = (p,Kp-q)$ if $p/q > 1/K$
  \item $(p',q') = (q-Kp,p+K(q-Kp))$ if $p/q \leq 1/K$
\end{itemize}
(see \cite[Cor. 31]{Kol}).
Similarly, if $C$ is $(p,q)$-well-placed with respect to $(\calN;\calB_-,\calB_+)$, then $\Psi_X(C)$ is $(p',q')$-well-placed with respect to $(\calN;\calB_-,\calB_+)$, where
\begin{itemize}
  \item $(p',q') = (q+K(p-Kq),p-Kq)$ if $p/q \geq K$
  \item $(p',q') = (Kq-p,q)$ if $p/q \leq K$.
\end{itemize}
The effects of $\Phi_X,\Psi_X$ on the homology class of $C$ are also described in \cite{Kol} (c.f. Remark~\ref{rmk:alt_adj_arg}).

Let $\Sset_X$ be defined as in Theorem~\ref{thmlet:curves_in_X},
and let $\Sset_X^\calN \subset \Sset_X \subset [1,\infty)$ denote the set of reduced fractions $p/q \geq 1$ for which there exists a rational algebraic curve in $X$ which is $(p,q)$-well-placed with respect to $(\calN;\calB_-,\calB_+)$ or $(\calN;\calB_+,\calB_-)$.
Defining functions
\begin{align*}
\Ssym_X: (1,\infty) \ra (1,\infty),&\;\;\;\;\; \Ssym_X(x) = K - \tfrac{1}{x} \\
\Rsym_X: (K,\infty) \ra (K,\infty),&\;\;\;\;\; \Rsym_X(x) = K + \tfrac{1}{x-K},
\end{align*}
the above discussion shows that, for $X$ a del Pezzo surface of degree at least $4$, the sets $\Sset_X^\calN \cap (1,\infty)$ and $\Sset_X^\calN \cap (K,\infty)$ are preserved by $\Ssym_X$ and $\Rsym_X$ respectively. 
For $X$ rigid (i.e. $K+2 \geq 5$), let $a_\acc^X = \tfrac{1}{2}(K + \sqrt{K^2-4})$ denote the unique fixed point of $\Ssym_X$ lying in $(1,\infty)$.\footnote{Note that the other fixed point of $x \mapsto K - \tfrac{1}{x}$ is $\tfrac{1}{a_\acc} < 1$.}
This agrees with the staircase accumulation point of the corresponding ellipsoid embedding function in the case $K+2 \geq 5$, while for $K+2 = 4$ we have $a_\acc^X = 1$.
One can readily check the following:
\begin{enumerate}[label=(\roman*)]
  \item $K-1 < a_\acc^X < K$
  \item $(a_\acc^X,\infty) = \bigsqcup\limits_{i=0}^\infty \Ssym_X^i([K,\infty))$
  \item $\Ssym_X$ is strictly increasing on $(1,a_\acc^X)$ and strictly decreasing on $(a_\acc^X,\infty)$
  \item $\Rsym_X$ is an involution which fixes $K+1$ and exchanges $(K,K+1)$ with $(K+1,\infty)$.
\end{enumerate}

For $\calN$ a uninodal anticanonical divisor in a smooth complex projective surface $X$, let $\pp \in \calN$ be a smooth point, let $\wt{X}$ be the blowup of $X$ at $\pp$, and let $\wt{\calN} \subset \wt{X}$ be the strict transform of $\calN$.
Observe that if $C \subset X$ is $(p,q)$-well-placed with respect to $\calN$ then its strict transform $\wt{C} \subset \wt{X}$ is $(p,q)$-well-placed with respect to $\wt{\calN}$.
In particular, we have 
\begin{align}\label{eq:blowup_subset}
\Sset_X^\calN \subset \Sset_{\wt{X}}^{\wt{\calN}}.
\end{align}

\begin{prop}\label{prop:Thm_F_holds} Theorem~\ref{thmlet:curves_in_X} holds for all $X \neq \CP^1 \times \CP^1$.
\end{prop}
\begin{proof}
Suppose that $X$ is the del Pezzo surface given by blowing up $\CP^2$ at points $\pp_1,\dots,\pp_j \in \CP^2$ for some $j \in \{0,\dots,8\}$, where no $3$ of the points line on a line and no $6$ lie on a conic.
Let $\calN_X \subset X$ be any given uninodal anticanonical divisor, and let $\calN \subset \CP^2$ be its image under the blowdown map $X \ra \CP^2$.
Noting that $\calN$ necessarily passes through $\pp_1,\dots\,\pp_j$, let $\pp_{j+1},\dots,\pp_8$ be some additional distinct points in $\calN$. 
For $i = 1,\dots,8$, let $\bl^i\CP^2$ denote the blowup of $\CP^2$ at $\pp_1,\dots,\pp_i$, let $\calN^{(i)} \subset \bl^i\CP^2$ denote the strict transform of $\calN$, and let $\Sset^{\calN^{(i)}}_{\bl^i\CP^2} \subset \Sset_{\bl^i\CP^2}$ denote the set of reduced fractions $p/q \geq 1$ for which there exists a rational algebraic curve in $\bl^i\CP^2$ which is $(p,q)$-well-placed with respect to $\calN^{(i)}$ (for some labeling of the branches near the double point). Note that here we have $\Sset_{\bl^j\CP^2}^{\calN^{(j)}} = \Sset_{X}^{\calN_X}$.

It follows by Theorem~\ref{thmlet:sesqui_plane_curves} that $\Sset_{\CP^2}^{\calN}$ is dense in $[a_\acc^{\CP^2},\infty)$, where $a_\acc^{\CP^2} = \tau^4 \approx 6.85$.
According to \eqref{eq:blowup_subset}, we have $\Sset_{\CP^2}^\calN \subset \Sset_{\bl^1\CP^2}^{\calN^{(1)}}$, so the latter is also dense in $[a_\acc^{\CP^2},\infty)$.
By applying the symmetry $\Rsym_X$, it follows that $\Sset_{\bl^1\CP^2}^{\calN^{(1)}}$ is dense in $(6,\infty)$, and 
hence also in $[a_\acc^{\bl^1\CP^2},\infty)$ by repeatedly applying the symmetry $\Ssym_X$.
Similarly, we have $\Sset_{\bl^1\CP^2}^{\calN^{(1)}} \subset \Sset_{\bl^2\CP^2}^{\calN^{(2)}}$, so the latter is dense in $[a_\acc^{\bl^1\CP^2},\infty)$, and hence also in $[5,\infty)$ by applying the symmetries $\Ssym_X$ and $\Rsym_X$.
Continuing inductively in this manner, we find that $\Sset_{\bl^i\CP^2}^{\calN^{(i)}}$ is dense in $[a_\acc^{\bl^i\CP^2},\infty)$ for $i = 0,\dots,4$, and in $[1,\infty)$ for $i = 5$.
If $j \leq 5$ we are done, while for $j \geq 6$ we have $\Sset_{\bl^5\CP^2}^{\calN^{(5)}} \subset \Sset_{\bl^j\CP^2}^{\calN^{(j)}}$ (again by \eqref{eq:blowup_subset}), and hence the latter is also dense in $[1,\infty)$.

Finally, to prove the claim about curves below the accumulation point, let $X$ be a unimonotone rigid del Pezzo surface, and note that we have $c_{X}(a) = \tfrac{a}{a+1}$ whenever $a$ is the $x$-value of an outer corner of the infinite staircase $c_X|_{[1,a_\acc^X]}$, while $c_X(a) < \tfrac{a}{a+1}$ for all other values of $a \in [1,a_\acc^X]$.
At the same time, for each $p/q \in \Sset_X \cap [1,a_\acc^X]$ we have
$c_X(p/q) \geq \tfrac{p/q}{p/q+1}$ by ~\eqref{eq:HK_lb}, and hence $p/q$ must be the $x$-value of an outer corner.
\end{proof}
\begin{rmk}\label{rmk:gen_rat_surf_II}
Following up on Remark~\ref{rmk:jgeq9case}(1), 
let $X$ be a del Pezzo surface, and suppose that $C$ is a rational algebraic curve in $X$ which is $(p,q)$-well-placed with respect to a uninodal anticanonical divisor $\calN \subset X$.
Let $\wt{X}$ be the blowup of $X$ at a finite set of points $\pp_1,\dots,\pp_k$ in $X \setminus C$, and let $\wt{C}$ denote the strict transform of $C$ in $\wt{X}$.
Then $\wt{C}$ still has a $(p,q)$ cusp and satisfies $\ind_\C^{p,q}(\wt{C}) = 0$, whence $p/q \in \Sset_{\wt{X}}$. 
In particular, the condition $\pp_1,\dots,\pp_k \in X \setminus C$ evidently holds if $\pp_1,\dots,\pp_k$ lie in the smooth locus of $\calN$, or if we assume that $\pp_1,\dots,\pp_k$ are disjoint from the space of $(p,q)$-well-placed rational algebraic curves in $X$.
Since the latter space contains no positive dimensional families by Lemma~\ref{lem:no_pos_dim_fams},
it follows as in the proof above that $\Sset_X$ is dense in $[1,\infty)$ whenever $X$ is a blowup of $\CP^2$ at $k \geq 5$ points in very general position\footnote{Here we say that a condition holds for points $\pp_1,\dots,\pp_k \in X$ in \hl{very general position} if the subspace of $X^{\times k}$ for which it fails is contained in a countable union of proper subvarieties.} (in fact, in this case $\Sset_X = [1,\infty) \cap \Q$ by Remark~\ref{rmk:precise_density}).
\end{rmk}

\begin{rmk}\label{rmk:alt_adj_arg}

Notice that the last paragraph in the above proof of Proposition~\ref{prop:Thm_F_holds}
relies on knowledge of symplectic embeddings in order to rule out algebraic curves. 
Alternatively, one can also characterize curves below the accumulation point using only the adjunction formula and some combinatorics, as we now briefly sketch. 
Let $X$ be a rigid del Pezzo surface of degree $K+2 \in \{5,\dots,9\}$, and let $C$ be a rational algebraic curve in $X$ with a $(p,q)$ cusp such that $\ind^{p,q}_\C(A) := c_1(A) - p - q  = 0$, where $A := [C] \in H_2(X)$.
The adjunction formula reads $c_1(A) = 2 + A \cdot A - (p-1)(q-1) - 2\de$, 
where $\de \in \Z_{\geq 0}$ is the count of singularities away from the $(p,q)$ cusp.
One can show\footnote{For instance, this follows directly from the so-called light cone inequality as discussed in \cite[Prob. 4.5]{steele2004cauchy}.} that for any $B_1,B_2 \in H_2(X)$ with nonnegative self-intersection numbers, we have $(B_1 \cdot B_1)(B_2 \cdot B_2) \leq (B_1 \cdot B_2)^2$, and applying this inequality with $B_1 = A$ and $B_2 = - \calK_X$ (the anticanonical class) gives
\begin{align}
 p^2  - Kqp + q^2 + (K+2)(1-2\de) \geq 0.
\end{align}
We now consider the ``mutation'' 
\begin{align}
(p',q';A') := \FF_X(p,q;A) := (q,Kq-p;-q\calK_X-A),
\end{align}
which one can check preserves the above index zero and adjunction conditions (with the same values of $K$ and $\de$).
Moreover, by repeatedly applying $\FF_X$ we can iteratively decrease $q$ until we eventually arrive at a new pair $(p_0,q_0) \neq (K-1,1)$ such that either
\begin{multicols}{2}
\begin{enumerate}[label=(\roman*)]
  \item $q_0 = 1$ and $\de =0$, or
  \item $q_0=2$, $\de = 0$, and $K=3$.
\end{enumerate}
\end{multicols}
Thus $(p,q,A) = \FF_X^{-j}(p_0,q_0,A_0)$ for some $j \in \Z_{\geq 0}$, and one can check that the triples $(p_0,q_0,A_0)$ which are ``minimal'' as above correspond precisely to the seeds considered in \cite[\S2.4]{mcduff2024singular}, which generate the outer corner curves by successive applications of $\Phi_X$.
In particular, there are precisely $J$ seed pairs $(p_0,q_0)$, where $J=3$ when $X$ is $\bl^1\CP^2$ or $\bl^2\CP^2$
  and $J = 2$ in the remaining cases.
\end{rmk}

\begin{rmk}
The combinatorial argument in Remark~\ref{rmk:alt_adj_arg} is closely analogous to the one appearing in \cite[\S4]{gross2010quivers} to characterize the discrete part of basic scattering diagrams. 
If we restrict to those curves which are well-placed, then the numerical symmetry $\FF_X$ is geometrically realized by $\Phi_X$, after precomposing by the map $(p,q) \mapsto (q,p)$ corresponding to swapping the branches $\calB_-,\calB_+$. 
If we further restrict to the two-stranded cases (i.e. $J=2$), then the last sentence of Theorem~\ref{thmlet:curves_in_X}(a) is in fact equivalent to \cite[Thm. 5]{gross2010quivers} under our fundamental bijection (see \S\ref{subsec:wp_and_basic_sd} below).
\end{rmk}

\begin{rmk}\label{rmk:adj_not_compl_obs}
Fix a smooth complex projective surface $X$, and consider coprime $p,q \in \Z_{\geq 1}$ and a homology class $A \in H_2(X)$ satisfying the index zero condition $p+q = c_1(A)$.
Then the adjunction inequality $A \cdot A - c_1(A) + 2 \geq (p-1)(q-1)$ is a necessary condition for the existence of a rational algebraic curve with a $(p,q)$-cusp in homology class $A$, and in the case $X = \CP^2$ it follows from Theorem~\ref{thmlet:sesqui_plane_curves} that this is also a sufficient condition.

However, for other del Pezzo surfaces the situation is more subtle.
For example, in the first Hirzebruch surface $F_1 = \CP^2 \# \ovl{\CP}^2$, the ``fake perfect exceptional'' homology class $A = 9\ell - 5e_1$ 
satisfies the index zero and adjunction conditions for $(p,q) = (19,3)$, but there does not exist any rational algebraic curve $C$ in $F_1$ with a $(19,3)$ cusp such that $[C] = A$ (such a curve would necessarily be unicuspidal).
Indeed, if such a curve existed, then its minimal normal crossing resolution would be a nonsingular curve $\wt{C}$ in a $10$-point blowup of $\CP^2$ with $[\wt{C}] = 9\ell - 5e_1 - 3e_2 - \cdots - 3e_7 - e_8 - e_9 - e_{10}$, which impossibly intersects the cubic class $3\ell - 2e_1 - e_2 - \cdots - e_7$ negatively
(see \cite[Rmk. 2.1.15]{magill2022staircase} for more details). 
At the same time, this argument does {\em not} rule out a rational algebraic curve in $F_1$ with a $(19,3)$ cusp lying in homology class $8\ell - 2e_1$, and this is consistent with Theorem~\ref{thmlet:curves_in_X} since we have $19/3 > a_\acc^{F_1} \approx 5.83$ (such a curve necesarrily has $\de_{19,3} = 2$ singularities).
\end{rmk}

\section{The fundamental bijection}\label{sec:fund_bij}

In this section, we show that any uninodal Looijenga pair $(X,\calN)$ induces a bijection between (a) well-placed curves in $X$, and (b) curves in an edge blowup of a certain singular toric surface which intersect the preferred anticanonical divisor in one point.
Here the curves in (a) are our main geometric object of interest, while the curves in (b) are fruitfully encoded using scattering diagrams (as we discuss in the next section).
After explaining the relevant notions and terminology in \S\ref{subsec:tor_mod_prelim}, we formulate the general bijection abstractly in \S\ref{subsec:tor_mod_fund_bij}.
Finally, in \S\ref{subsec:rigid_dP_tor_mods} we give explicit small toric models for each of the rigid del Pezzo surfaces, and we compute the corresponding bijections.

\subsection{Toric models for uninodal Looijenga pairs}\label{subsec:tor_mod_prelim}

Following e.g. \cite{gross2015mirror}, a \hl{Looijenga pair} $(X,\Ddiv)$ is a smooth complex projective surface $X$ together with a nodal anticanonical divisor $\Ddiv$ with at least one node. 
For example:
\begin{itemize}
  \item If $V^\tor$ is a smooth toric surface with toric boundary divisor $\Ddiv^\tor$, then $(V^\tor,\Ddiv^\tor)$ is a Looijenga pair, which we call a \hl{toric Looijenga pair}. 

  \item If $(X,\Ddiv)$ is a Looijenga pair, $X_+$ is the blowup of $X$ at one or more nodes of $\Ddiv$, and $\Ddiv_+ \subset X_+$ is the total transform of $\Ddiv$, then $(X_+,\Ddiv_+)$ is again a Loojienga pair, which we call a \hl{corner blowup} of $(X,\Ddiv)$.

    \item If $(X,\Ddiv)$ is a Looijenga pair, $X[\calS]$ is the blow up of $X$ along a set of smooth points $\calS \subset \Ddiv$, and $\Ddiv[\calS] \subset X[\calS]$ is the strict transform of $\Ddiv$, then $(X[\calS],\Ddiv[\calS])$ is again a Looijenga pair, which we call an \hl{edge blowup} of $(X,\Ddiv)$.

\end{itemize}

We will say that a Looijenga pair $(X,\Ddiv)$ is \hl{uninodal} if $\Ddiv$ has exactly one node (or equivalently $\Ddiv$ is irreducible); we will often denote uninodal anticanonical divisors by $\calN$.

\begin{definition}\label{def:tor_mod}
A \hl{toric model} $\tormod$ for a uninodal Looijenga pair $(X,\calN)$ is:
\begin{itemize}
  \item a sequence $(X^{(0)},\calN^{(0)}),\dots,(X^{(k)},\calN^{(k)})$, where $(X^{(0)},\calN^{(0)}) = (X,\calN)$ and, for $j = 1,\dots,k$, $X^{(j)}$ is the blowup of $X^{(j-1)}$ at some node $\pp^{(j-1)} \in \calN^{(j-1)}$, with $\calN^{(j)} \subset X^{(j)}$ the strict transform of $\calN^{(j-1)}$
  \item $\E_1,\dots,\E_\ell \subset X^{(k)}$ pairwise disjoint smoothly embedded rational curves with self-intersection number $-1$

  \item a toric Looijenga pair $(V^\tor,\Ddiv^\tor)$

  \item a birational morphism $X^{(k)} \ra V^\tor$ sending $\calN^{(k)}$ to $\Ddiv^\tor$ which has exceptional divisors $\E_1,\dots,\E_\ell$ lying over smooth points $\calS \subset \Ddiv^\tor$ and is otherwise an isomorphism.

\end{itemize}
\end{definition}

\NI In particular, $(X^{(k)},\calN^{(k)})$ is isomorphic to the edge blowup $(V^\tor[\calS],\Ddiv^\tor[\calS])$ of $(V^\tor,\Ddiv^\tor)$ along $\calS$.
Note that there are $j$ possibilities for the node $\pp^{(j-1)} \in \calN^{(j-1)}$, and in particular $\pp^{(0)}$ is the unique node of $\calN$.

The following notation will be useful in the sequel:
\begin{notation}\label{not:tormod_m_and_l}
Let $\tormod$ be a toric model for a uninodal Looijenga pair $(X,\Ddiv)$ as in Definition~\ref{def:tor_mod}, and assume that $V^\tor$ has associated fan\footnote{See e.g. \cite{cox2011toric,fulton1993introduction} for the definition of fans and other standard terminology from toric algebraic geometry.} $\Si$ in $\MM_\R := \MM \otimes_\Z \R$ for some rank two lattice $\MM$.
Let
\begin{itemize}
    \item $-\mm_1,\dots,-\mm_J \in \MM$ be the primitive generators of those rays of $\Si$ which correspond to a divisor of $V^\tor$ containing at least one point in $\calS$.

    \item $\ell_j$ be the number of points in $\calS$ lying on the toric divisor associated with the ray $\R_{\geq 0} \cdot (-\mm_j)$, for $j=1,\dots,J$.
  \end{itemize}
\end{notation}
\NI  Thus $\ell = \sum\limits_{j=1}^J \ell_j$, where $\E_1,\dots,\E_\ell$ are the exceptional divisors of $X^{(k)} \ra V^\tor$. 

\begin{example}
Our primary example of a uninodal Looijenga pair is $(\CP^2,\calN_0)$, where $\calN_0$ is the nodal cubic $\{x^3 + y^3 = xyz\} \subset \CP^2$.
In \S\ref{subsec:rigid_dP_tor_mods}, we describe our preferred toric model $\tormod_{\CP^2}$ for which $V^\tor = F_3$ is the third Hirzebruch surface.
Here we view the fan $\Si$ as having ray generators $(0,-1),(1,3),(0,1),(-1,0)$ given by the primitive outward normals to the corresponding moment polygon with vertices $(0,0),(4,0),(1,1),(0,1)$. For this toric model we have $J = 2$, $\mm_1 = (1,0)$, $\mm_2 = (-1,-3)$, $\ell = 2$, and $\ell_1 = \ell_2 = 1$ (c.f. Table~\ref{table:fund_bij}).
\end{example}

Abstractly, any uninodal Looijenga pair $(X,\calN)$ admits a toric model by \cite[Prop 1.3]{gross2015mirror}. In fact, there are typically many inequivalent toric models, and we will find it fruitful to seek toric models with $J$ as small as possible, as this will mean the associated scattering diagram has few initial rays (see \S\ref{sec:from_wp_to_sd_and_back}).

\subsection{The fundamental bijection}\label{subsec:tor_mod_fund_bij}

Let $(X,\calN)$ be a uninodal Looijenga pair with a toric model $\tormod$ as in Definition~\ref{def:tor_mod}. 
We assume that the corresponding toric Looijenga pair $(V^\tor,\Ddiv^\tor)$ has associated fan $\Si$ in $\MM_\R$ for some rank two lattice $\MM$, i.e. $V^\tor = V_\Si$.
For nonzero $\mm \in \MM$, let $\gcd_\MM(\mm)$ denote the largest $j \in \Z_{\geq 1}$ such that $\mm \in j\MM$ (this is just the usual greatest common divisor of the components of $\mm$ when $\MM = \Z^2$). 
Below we will define the following data associated to $\tormod$:
\begin{itemize}
  \item for each nonzero $\mm \in \MM$, a Looijenga pair $(V^\tor_{+\mm}[\calS],\Ddiv^\tor_{+\mm}[\calS])$ which is typically an edge blowup of a (weighted\footnote{Here \hl{weighted blowup} at a point means concretely a birational transformation modeled on the passage from $V_\Si$ to $V_{\Si'}$, where $V_{\Si},V_{\Si'}$ are toric varieties modeled on complete fans $\Si,\Si'$ respectively, such that $\Si'$ refines $\Si$ by adding one ray. As explained in  \cite[\S3.1]{mcduff2024singular}, symplectically this corresponds to an ellipsoidal blowup.}) corner blowup of $(V^\tor,\Ddiv^\tor)$

  \item a distinguished irreducible component $\Ddiv_\out \subset \Ddiv^\tor_{+\mm}[\calS]$

  \item a piecewise linear map $\wp_X: \Z^2_{\geq 0} \ra \MM$ 

\end{itemize}
such that the following holds: 

\begin{prop}\label{prop:fund_bij}
For each $p,q \in \Z_{\geq 0}$ not both zero, there is a bijection between
\begin{enumerate}[label=(\alph*)]
\item simple rational algebraic curves in $X$ which are $(p,q)$-well-placed with respect to $\calN$, and 
  \item simple rational algebraic curves in $V^\tor_{+\wp_X(p,q)}[\calS]$ which intersect $\Ddiv_\out$ in one point with contact order $\gcd(p,q)$ and are otherwise disjoint\footnote{In particular, this means that the unique intersection point of such a curve with $\Ddiv_\out$ lies in the complement of the other components of $\Ddiv^\tor_{+\wp_X(p,q)}[\calS]$.} from $\Ddiv^\tor_{+\wp_X(p,q)}[\calS]$. 
\end{enumerate}
Moreover, $\wp_X$ descends to a bijection from $\nicequot{\Z^2_{\geq 0}}{\sim}$ to $\MM$, where we put $(j,0) \sim (0,j)$ for each $j \in \Z_{\geq 1}$, and we have $\gcd(p,q) = \gcd_\MM(\wp_X(p,q))$.
\end{prop}
\NI 
Here ``simple rational algebraic curve in $X$'' means the image $C$ of a holomorphic map $\CP^1 \ra X$ which does not factor through a holomorphic map $\CP^1 \ra \CP^1$ of degree $\geq 2$ (in particular $C$ is reduced and irreducible). 
Since such curves have a unique holomorphic parametrization up to biholomorphism, we will tend to view $C$ as both a map $\CP^1 \ra X$ and as a subvariety of $X$.
Note that simplicity in (a) and (b) is vacuous if $\gcd(p,q) = 1$.

\begin{rmk}
Note that the special cases with $p = 0$ or $q = 0$ are to be interpreted as in Convention~\ref{conv:pq=0}.
 Another noteworthy special case occurs when $\wp_X(p,q)$ is positively proportional to $-\mm_i$ for some $i \in \{1,\dots,J\}$, in which case the exceptional divisors $\E_1,\dots,\E_\ell$ themselves contribute to (b).
\end{rmk}

\sss

To begin, for each nonzero $\mm \in \MM$ not lying on any ray of $\Si$, let $\Si_{+\mm}$ denote the refinement of the fan $\Si$ obtained by adding the ray $\R_{\geq 0} \cdot \mm$ (and appropriately subdividing the cone containing it). 
We denote by $\Ddiv_\out$ the toric divisor in the associated (typically singular) toric surface $V^\tor_{+\mm} := V_{\Si_{+\mm}}$ corresponding to the new ray $\R_{\geq 0} \cdot \mm$.
Note that $V^\tor_{+\mm}$ is a weighted blowup of $V^\tor$, and there is an induced birational morphism $V^\tor_{+\mm} \ra V^\tor$.
On the other hand, if $\mm \in \MM$ is nonzero and lies on a ray of $\Si$, then we put simply $V^\tor_{+\mm} := V^\tor$ and let $\Ddiv_\out \subset V^\tor_{+\mm}$ denote the toric divisor corresponding to the ray $\R_{\geq 0} \cdot \mm$.

As in Definition~\ref{def:tor_mod}, the image of $\E_1 \cup \cdots \cup \E_\ell$ in $V^\tor$ is a finite set $\calS$ of smooth points in $\Ddiv^\tor$.
Let $V^\tor[\calS]$ denote the blowup of $V^\tor$ along $\calS$,
so that the birational morphism $X^{(k)} \ra V^\tor$ from Definition~\ref{def:tor_mod} lifts to an isomorphism $X^{(k)} \ra V^\tor[\calS]$.
Let $V^\tor_{+\mm}[\calS]$ denote the blowup of $V^\tor_{+\mm}$ along the preimage of $\calS$ under $V^\tor_{+\mm} \ra V^\tor$,
and let $X^{(k)}_{+\mm}$ be the corresponding weighted blowup of $X^{(k)}$ fitting into the following commutative diagram:
\[
\begin{tikzcd}
  X_{+\mm}^{(k)} \arrow[r,"\cong"] \arrow[d] &  V^\tor_{+\mm}[\calS] \arrow[d]\arrow[r] & V^\tor_{+\mm} \arrow[d]\\
  X^{(k)} \arrow[r,"\cong"] & V^\tor[\calS] \arrow[r] & V^\tor.
\end{tikzcd}
\]
Let $\Ddiv^\tor_{+\mm}$ denote the toric boundary divisor of $V^\tor_{+\mm}$, and let $\Ddiv_{+\mm}^\tor[\calS]$ denote its strict transform in $V^\tor_{+\mm}[\calS]$.

We now define $\wp_X(p,q)$ for $p,q \in \Z_{\geq 1}$ as follows.
Recall that as part of the data of the toric model $\tormod$ we have nodes $\pp^{(j)} \in X^{(j)}$ for $j=0,\dots,k-1$, where $X^{(j)}$ is the blowup of $X^{(j-1)}$ at $\pp^{(j-1)}$.
Let $\calB_-^{(0)} := \calB_-$ and $\calB_+^{(0)} := \calB_+$ denote the local branches of $\calN$ near its node $\pp^{(0)}$, and put $(p^{(0)},q^{(0)}) := (p,q)$.
Next, for $j=1,\dots,k$:
\begin{enumerate}[label=(\roman*)]
  \item if $\pp^{(j-1)} \neq \calB_-^{(j-1)} \cap \calB_+^{(j-1)}$, let $\calB^{(j)}_\pm$ denote the strict transform $\calB^{(j-1)}_\pm$ and put $(p^{(j)},q^{(j)}) := (p^{(j-1)},q^{(j-1)})$

  \item if $\pp^{(j-1)} = \calB_-^{(j-1)} \cap \calB_+^{(j-1)}$ and $p^{(j-1)} > q^{(j-1)}$, let $\calB_-^{(j)}$ denote the strict transform of $\calB_-^{(j-1)}$ in $X^{(j)}$, let $\calB_+^{(j)}$ denote the exceptional divisor of $X^{(j)} \ra X^{(j-1)}$, and 
put $(p^{(j)},q^{(j)}) := (p^{(j-1)} - q^{(j-1)},q^{(j-1)})$

  \item if $\pp^{(j-1)} = \calB_-^{(j-1)} \cap \calB_+^{(j-1)}$ and $q^{(j-1)} > p^{(j-1)}$, let $\calB_+^{(j)}$ denote the strict transform of $\calB_+^{(j-1)}$ in $X^{(j)}$, let $\calB_-^{(j)}$ denote the exceptional divisor of $X^{(j)} \ra X^{(j-1)}$, and
    put $(p^{(j)},q^{(j)}) := (p^{(j-1)},q^{(j-1)} - p^{(j-1)})$.

  \item otherwise, if $\pp^{(j-1)} = \calB_-^{(j-1)} \cap \calB_+^{(j-1)}$ and $p^{(j-1)} = q^{(j-1)}$, let both $\calB_-^{(j)}$ and $\calB_+^{(j)}$ denote the exceptional divisor of $X^{(j)} \ra X^{(j-1)}$, and put $(p^{(j)},q^{(j)}) := (p^{(j-1)},p^{(j-1)})$ (in this case every subsequent case is necessarily of type (i)).

\end{enumerate}

Let $\mm_-,\mm_+ \in \MM$ denote the primitive ray generators in $\Si$ corresponding to the toric boundary divisors of $V^\tor$ 
containing the images of $\calB_-^{(k)},\calB_+^{(k)}$ respectively under the composition $X^{(k)} \ra V^\tor[\calS] \ra V^\tor$.
Note that $\mm_- = \mm_+$ if and only if case (iv) above occurs at some step, and this can only hold for finitely many initial values of $p/q$.

\begin{definition}\label{def:W}
 For $p,q \in \Z_{\geq 1}$, put
\begin{align*}
\wp_X(p,q) := 
\begin{cases}
p^{(k)} \cdot \mm_- + q^{(k)} \cdot \mm_+ &\text{if}\;\; \mm_- \neq \mm_+ \\
p^{(k)} \cdot \mm_- & \text{if}\;\;\mm_- = \mm_+.
\end{cases}
\end{align*}
We also put $\wp_X(0,0) := (0,0)$, and $\wp(j,0) := \wp(0,j) := j\mmsubN$ for $j \in \Z_{\geq 1}$, where $\mmsubN \in \MM$ is the primitive ray generator of $\Si$ corresponding to the image of $\calN^{(k)}$ under the composition $X^{(k)} \ra V^\tor[\calS] \ra V^\tor$.

\end{definition}

\begin{proof}[Proof of Proposition~\ref{prop:fund_bij}]
 
Suppose first that $p,q \geq 1$ and $\mm_- \neq \mm_+$. 
Given a curve in $X$ which is $(p,q)$-well-placed with respect to $(\calN;\calB_-,\calB_+)$, a straightforward induction shows that its strict transform in $X^{(k)}$ is $(p^{(k)},q^{(k)})$-well-placed with respect to $(\calN^{(k)};\calB_-^{(k)},\calB_+^{(k)})$ (see \S\ref{subsec:rigid_dP_tor_mods} for explicit examples).
Conversely, if a curve in $X^{(k)}$ is $(p^{(k)},q^{(k)})$-well-placed with respect to $(\calN^{(k)};\calB_-^{(k)},\calB_+^{(k)})$, then its image under $X^{(k)} \ra X$ is $(p,q)$-well-placed with respect to $(\calN;\calB_-,\calB_+)$.
Furthermore, by construction the birational morphism $X^{(k)}_{+\wp_X(p,q)} \ra X^{(k)}$ together with the isomorphism $X^{(k)}_{+\wp_X(p,q)} \cong V^\tor_{+\wp_X(p,q)}[\calS]$ give a bijection between $(p^{(k)},q^{(k)})$-well-placed curves in $X^{(k)}$ and curves in $V^{\tor}_{+\wp_X(p,q)}[\calS]$ which intersect $\Ddiv_\out$ in one point with contact order $\gcd(p^{(k)},q^{(k)}) = \gcd(p,q)$ and are otherwise disjoint from $\Ddiv_{+\wp_X(p,q)}^\tor[\calS]$.

Now suppose that $p,q \geq 1$ and $\mm_- = \mm_+$. Similar reasoning shows that a curve is $(p,q)$-well-placed with respect to $(\calN;\calB_-,\calB_+)$ if and only if its strict transform in $X^{(k)}$ intersects $\calB_-^{(k)} = \calB_+^{(k)}$ in one point with contact order $\gcd(p,q)$ and is otherwise disjoint from $\calN^{(k)}$.
Under the isomorphism $X^{(k)} \cong V^\tor[\calS]$,
this corresponds to a curve in $V^\tor_{+\wp_X(p,q)}[\calS] = V^\tor[\calS]$ which intersects $\Ddiv_\out$ in one point with contact order $\gcd(p,q)$ and is otherwise disjoint from $\Ddiv^\tor_{+\wp_X(p,q)}[\calS]$.

Similarly, in the cases $(p,q) = (j,0)$ and $(p,q) = (0,j)$ for some $j \in \Z_{\geq 1}$, a curve in $X$ is $(p,q)$-well-placed if and only if its strict transform in $X^{(k)}$ intersects $\calN^{(k)}$ in one point smooth point of $\calN^{(k)}$ with contact order $j$ and is otherwise disjoint from $\calN^{(k)}$. As above, 
under the isomorphism $X^{(k)} \cong V^\tor[\calS]$ this corresponds to a curve in $V^\tor_{+\wp_X(p,q)}[\calS] = V^\tor[\calS]$ which intersects $\Ddiv_\out$ in one point with contact order $j$ and is otherwise disjoint from $\Ddiv^\tor_{+\wp_X(p,q)}[\calS]$.

At the same time, if for some nonzero $\mm \in \MM$ we are given a curve $C$ in $V^\tor_{+\mm}[\calS]$ which intersects $\Ddiv_\out$ in one point with contact order $\gcd_\MM(\mm)$ and is otherwise disjoint from $\Ddiv^\tor_{+\mm}[\calS]$, then its image $C'$ under the composition $V^\tor_{+\mm}[\calS] \cong X^{(k)}_{+\mm} \ra X$ is $(p',q')$-well-placed with respect to $\calN$ for some $p',q' \in \Z_{\geq 0}$ not both zero. 
Note that $V^\tor_{+\mm}[\calS] \ra X$ contracts every component of $\Ddiv^\tor_{+\mm}[\calS]$ except for the one corresponding to the ray $\R_{\geq 0} \cdot \mmsubN$ to the node of $\calN$. 
Also, if $C$ is merely a germ of a curve in $V^\tor_{+\mm}[\calS]$ which intersects $\Ddiv_\out$ in one point, then we can view $C'$ as a curve germ in $X$ which is $(p',q')$ well-placed with respect to $\calN$. 
Therefore, if $\mm$ is positively proportional to $\mmsubN$, then $C' \cap \calN$ is a smooth point of $\calN$, whence $p' = 0$ or $q' = 0$, and otherwise we have $p',q' \geq 1$. 
Thus the association $\mm \mapsto (p',q')$ gives a well-defined map $\Z^2 \ra \nicequot{\Z^2_{\geq 0}}{\sim}$ which is inverse to $\wp_X(p,q)$ as a map from $\nicequot{\Z^2_{\geq 0}}{\sim}$ to $\MM^2$.
Lastly, it is easy to check that each step of $\wp_X$ preserves the greatest common divisor of its inputs, which implies the equality $\gcd(p,q) = \gcd_\MM(\wp_X(p,q))$.

\end{proof}

\begin{rmk}\label{rmk:Si_red_and_smth_vers}\hfill
\begin{enumerate}
  \item 
 Let $\Si_\op{red}$ be any complete
 fan in $\MM_\R$ which contains the rays $\R_{\geq 0} \cdot (-\mm_i)$ for $i=1,\dots,J$,
and let $V^\tor_{+\mm,\op{red}}[\calS]$ be defined in the same way as $V^\tor_{+\mm}[\calS]$ but using $\Si_\op{red}$ instead of $\Si$.
Then Proposition~\ref{prop:fund_bij} still holds if we replace $V^\tor_{+\wp_X(p,q)}[\calS]$ with $V^\tor_{+\wp_X(p,q),\op{red}}[\calS]$. 

\item
Proposition~\ref{prop:fund_bij} also remains true if we replace $V^\tor_{+\mm}[\calS]$ by its nonsingular resolution given by refining the fan $\Si_{+\mm}$ in a minimal way so as to resolve the singularities introduced adding the ray $\R_{\geq 0} \cdot \mm$. Thus we could work with only standard blowups instead of weighted ones, at the cost of more blowups (and more notation).
\end{enumerate}
\end{rmk}

\subsection{Explicit toric models for rigid del Pezzo surfaces}\label{subsec:rigid_dP_tor_mods}

Let $X$ be a rigid del Pezzo surface, i.e. $X$ is diffeomorphic to $\bl^j\CP^2 = \CP^2 \#^{\times j}\ovl{\CP}^2$ for some $j \in \{0,\dots,4\}$ or $\CP^1 \times \CP^1$ and is equipped with its unique Fano complex structure (up to biholomorphism).
In the following, given any uninodal anticanonical divisor $\calN \subset X$, we will give a prefered toric model $\tormod_X$ for the Looijenga pair $(X,\calN)$ and explicitly describe the corresponding function $\wp_X: \Z^2_{\geq 1} \ra \Z^2$ appearing in Proposition~\ref{prop:fund_bij}.
The toric model $\tormod_X$ turns out to be essentially independent of the choice of $\calN$ (up to deforming the locations of the edge blowup points $\calS \subset D^\tor$), so we often suppress $\calN$ from the notation and speak simply of the ``toric model associated to $X$''.

We begin by explicitly describing the bijection in Proposition~\ref{prop:fund_bij} in the case of the complex projective plane.
Figure~\ref{fig:tor_mod_CP2} illustrates the toric model $\tormod_{\CP^2}$, with $k=3$, $\ell=2$, and $V^\tor \cong F_3$ (the third Hirzebruch surface).
For $j = 0,1,2,3$ the solid curves represent the components of $\Ndiv^{(j)}$, while the dashed lines represent those curves which become $\E_1,\E_2$ in $X^{(3)}$. The labels give self-intersection numbers and homology classes, and the dot represents the node $\pp^{(j)} \in \Ndiv^{(j)}$ at which we blow up.

  \begin{figure}
  \centering
  \includegraphics[scale=1]{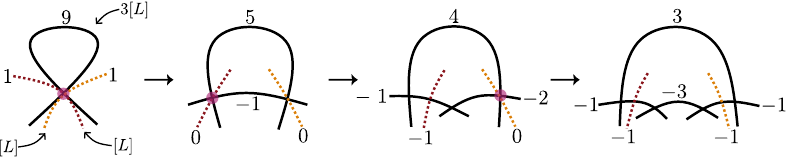}
  \caption{A toric model for $\CP^2$, with its (essentially unique) uninodal anticanonical divisor $\calN_{\CP^2}$. }
  \label{fig:tor_mod_CP2}
  \end{figure}

In more detail, let $C$ be a curve in $\CP^2$ which is $(p,q)$-well-placed with respect to $\calN$, with $p,q \in \Z_{\geq 1}$.
We denote this situation using the diagram
\begin{align*}
\calN \lineseg{p,q} \calN, 
\end{align*}
where the line segment represents the node $\pp^{(0)}$ of $\calN$.
Assume for the moment that $p/q \notin \{2,1,1/2\}$.
After blowing up at $\pp^{(0)}$ and denoting the resulting exceptional curve by $\F_1^{(1)}$, the diagram becomes:
  \begin{align*}
    \calN^{(1)} \lineseg{p-q,q} \F_1^{(1)} \lineseg{}\calN^{(1)}\;\;\;&\text{if}\;\;\; p > q\\
   \calN^{(1)} \lineseg{} \F_1^{(1)} \lineseg{p,q-p}\calN^{(1)} \;\;\; &\text{if}\;\;\;p < q.
  \end{align*}
Continuing by blowing up at the node $\pp^{(1)} \in \Ndiv$, we have:
\begin{align*}
   \calN^{(2)} \lineseg{p-2q,q} \F_2^{(2)} \lineseg{} \F_1^{(2)} \lineseg{} \calN^{(2)}  \;\;\;&\text{if}\;\;\; p > 2q \\
   \calN^{(2)} \lineseg{} \F_2^{(2)} \lineseg{p-q,2q-p} \F_1^{(2)} \lineseg{} \calN^{(2)} \;\;\;&\text{if}\;\;\;q < p < 2q \\ 
   \calN^{(2)} \lineseg{} \F_2^{(2)} \lineseg{} \F_1^{(2)} \lineseg{p,q-p} \calN^{(2)}\;\;\;&\text{if}\;\;\; p < q,
\end{align*}
where $\F_2^{(2)}$ is the new exceptional divisor and $\F_1^{(2)}$ is the strict transform of $\F_1^{(1)}$.
Finally, we blow up at the node $\pp^{(2)} \in \Ndiv^{(2)}$ to arrive at:
\begin{align*}
 \calN^{(3)} \lineseg{p-2q,q} \F_2^{(3)} \lineseg{} \F_1^{(3)} \lineseg{} \F_3^{(3)} \lineseg{} \calN^{(3)}   \;\;\;&\text{if}\;\;\; p > 2q \\
\calN^{(3)} \lineseg{} \F_2^{(3)} \lineseg{p-q,2q-p} \F_1^{(3)} \lineseg{} \F_3^{(3)} \lineseg{} \calN^{(3)} \;\;\;&\text{if}\;\;\; q < p < 2q  \\
\calN^{(3)} \lineseg{} \F_2^{(3)} \lineseg{} \F_1^{(3)} \lineseg{2p-q,q-p} \F_3^{(3)} \lineseg{} \calN^{(3)} \;\;\;&\text{if}\;\;\;  p < q < 2p \\
\calN^{(3)} \lineseg{} \F_2^{(3)} \lineseg{} \F_1^{(3)} \lineseg{} \F_3^{(3)} \lineseg{p,q-2p} \calN^{(3)} \;\;\;&\text{if}\;\;\; q > 2p.
\end{align*}

One readily checks using the classification of minimal rational surfaces that after blowing down the $(-1)$-curves $\E_1,\E_2 \subset X^{(3)}$ the result is isomorphic to $F_3$. 
Let $\Si$ be the fan in $\R^2$ for $F_3$ with ray generators $(0,-1),(1,3),(0,1),(-1,0)$, corresponding to $\calN,\F^{(3)}_2,\F^{(3)}_1,\F^{(3)}_3$ respectively.\footnote{Note that these are the outer normal vectors to the corresponding moment polygon with vertices $(0,0),(4,0),(1,1),(0,1)$. Taking instead the inner normal vectors would give a different but abstractly isomorphic fan.}
Note that we have (up to reordering) $\ell_1 = 1$, $\ell_2 = 1$, $\mm_1 = (1,0)$ and $\mm_2 = (-1,-3)$.
Then, according to Definition~\ref{def:W}, we have:
\begin{align}\label{eq:fundbijCP2}
\wp_{\CP^2}(p,q) = 
\begin{cases}
  (p-2q)(0,-1) + q(1,3) = (q,5q-p) & \text{if}\;\;\; p > 2q \\
  (p-q)(1,3) + (2q-p)(0,1) = (p-q,2p-q) & \text{if}\;\;\; q < p < 2q \\
  (2p-q)(0,1) + (q-p)(-1,0) = (p-q,2p-q) & \text{if}\;\;\; p < q < 2p \\
  p(-1,0) + (q-2p)(0,-1) = (-p,2p-q) & \text{if}\;\;\; q > 2p,
\end{cases},
\end{align}
where the two middle cases coalesce.
In the borderline cases $p/q = 2,1,1/2$, the strict transform of $C$ in $X^{(3)}$ intersects one of $\F_2^{(3)},\F_1^{(3)},\F_3^{(3)}$ in one point with contact order $\gcd(p,q)$ and is otherwise disjoint from the total transform of $\calN$, and one can check that the above formula for $\wp_{\CP^2}(p,q)$ is also valid in these cases.
Similarly, since $\mmsubN = (0,-1)$, in the cases $(p,q) = (j,0)$ or $(p,q) = (0,j)$ with $j \in \Z_{\geq 0}$ we have $\wp_{\CP^2}(p,q) = (0,-j)$.
Note that $p+q$ is divisible by $3$ if and only if $\wp_{\CP^2}(p,q)$ lies in $\Z_{\geq 0} \cdot \mm_1 + \Z_{\geq 0} \cdot \mm_2$.

\begin{rmk}
It turns out that there are no $(p,q)$-well-placed rational curves in $\CP^2$ for $p/q$ lying in $(1/5,1/2) \cup (1/2,2) \cup (2,5)$, so in particular both sides of the bijection in Proposition~\ref{prop:fund_bij} must be empty in this region.
Thus apart from the cases $p/q = 1/2,2$, which correspond under $\wp_{\CP^2}(p,q)$ to $\R_{\leq 0} \cdot \mm_1$ and $\R_{\leq 0} \cdot \mm_2$ respectively, 
 $\wp_{\CP^2}(p,q)$ lies in the cone generated by $\mm_1 = (1,0)$ and $\mm_2 = (-1,-3)$ whenever the bijection is nonvacuous. 
Note that the cases $p=0$ or $q=0$ correspond under $\wp_{\CP^2}$ to $\R_{\geq 0} \cdot (0,-1)$, which could be viewed as the ``center'' of $\cone(\mm_1,\mm_2)$. 
 We will see that the corresponding scattering diagram defined in \S\ref{sec:from_wp_to_sd_and_back} precisely matches this structure (c.f. Remark~\ref{rmk:scat_rays_lie_in_cone}).
\end{rmk}

The analogous computations for the remaining five rigid del Pezzo surfaces are similar. The toric models are illustrated in Figures ~\ref{fig:tor_mod_CP1xCP1},\ref{fig:tor_mod_Bl1},\ref{fig:tor_mod_Bl2},\ref{fig:tor_mod_Bl3},\ref{fig:tor_mod_Bl4}. The resulting bijections as in Proposition~\ref{prop:fund_bij} are summarized in Table~\ref{table:fund_bij}.

  \begin{figure}
  \centering    
  \includegraphics[scale=1]{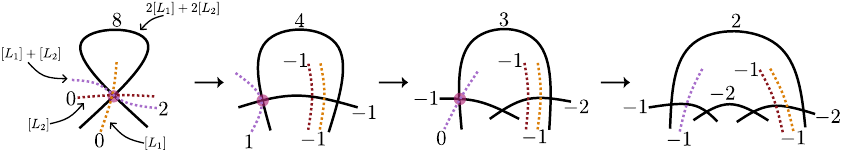}
  \caption{A toric model for $\CP^1 \times \CP^1$.} \label{fig:tor_mod_CP1xCP1}
  \end{figure}

    \begin{figure}
  \centering    
  \includegraphics[scale=1]{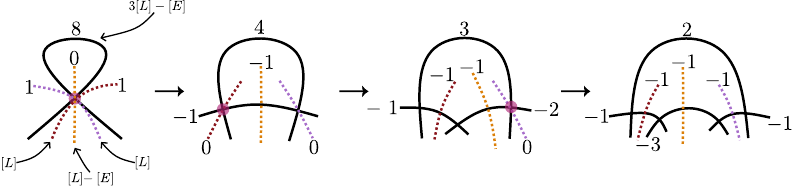}
  \caption{A toric model for $\bl^1\CP^2$.} \label{fig:tor_mod_Bl1}  
  \end{figure}

    \begin{figure}
  \centering    
  \includegraphics[scale=1]{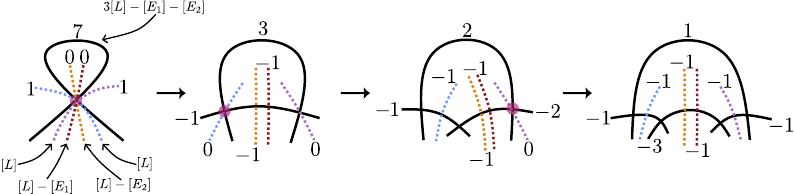}
  \caption{A toric model for $\bl^2\CP^2$.}  \label{fig:tor_mod_Bl2}
  \end{figure}

    \begin{figure}
  \centering    
  \includegraphics[scale=1]{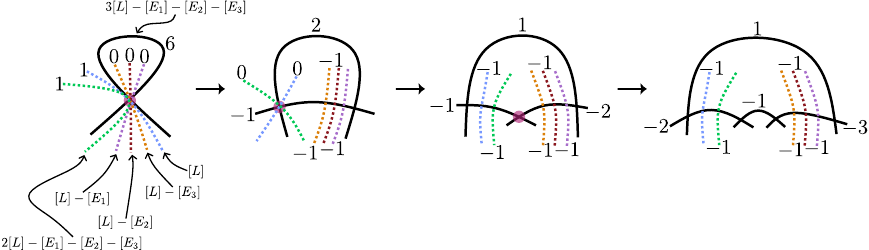}
  \caption{A toric model for $\bl^3\CP^2$.}  \label{fig:tor_mod_Bl3}
  \end{figure}

    \begin{figure}
  \centering    
  \includegraphics[scale=1]{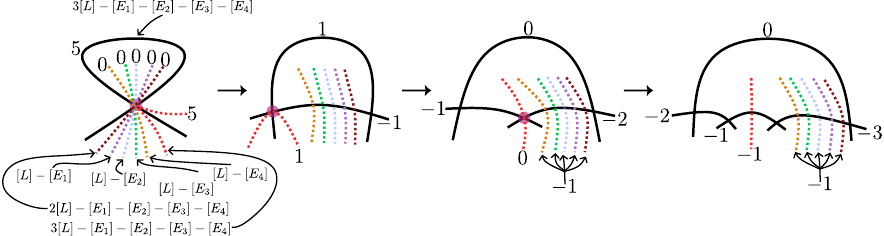}
  \caption{A toric model for $\bl^4\CP^2$.}  \label{fig:tor_mod_Bl4}
  \end{figure}

\begin{rmk}\label{rmk:atfs_and_tms}
Although we defined toric models in Definition~\ref{def:tor_mod} in terms of algebraic geometry, they turn out to be closely connected with the symplectic notion of almost toric fibrations (see e.g. \cite{symington71four,evans2023lectures,mcduff2024singular}). In fact, the toric models described above naturally correspond to the minimal almost toric fibrations pictured in Figure~\ref{fig:rigid_dP_ATFs}, which in turn control the infinite symplectic staircases in \cite{cristofaro2020infinite,casals2022full,mcduff2024singular} (and also e.g. exotic Lagrangian tori \cite{vianna2017infinitely}).
Comparing with Table~\ref{table:fund_bij}, we see that $J$ coincides with the number of nodal rays (i.e. eigenlines of the affine monodromy -- see e.g. \cite[\S4.1c]{mcduff2024singular}), with $\ell_1,\dots,\ell_J$ their multiplicities, and moreover $\mm_1,\dots,\mm_J$ are parallel to eigendirections up to a global $\op{GL}_2(\Z)$-transformation.
We will elaborate on the relationship between almost toric fibrations and toric models in a followup paper.
\end{rmk}
  \begin{figure}
\centering    
\includegraphics[scale=1]{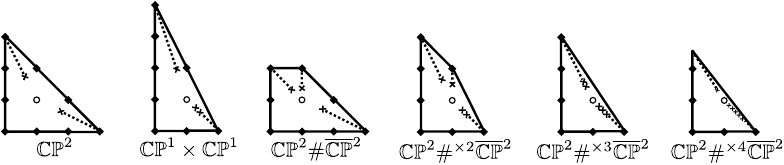}
\caption{Almost toric fibrations for the unimonotone rigid del Pezzo surfaces.}
\label{fig:rigid_dP_ATFs}  
\end{figure}

\begin{table} \caption{Toric models and the corresponding fundamental bijection for rigid del Pezzo surfaces.}\label{table:fund_bij} 

\scalebox{1}
{

\begin{tabular}{|c||c|c|c|c|c|c|c|}\hline 
$X$ & $V^\tor$
&   $\mm_1,\dots,\mm_J$ & $\wp_X(p,q)$ & $\ell_1,\dots,\ell_J$  \\
\hline\hline

$\CP^2$   & $F_3$   & $(1,0),(-1,-3)$ & $\begin{array}{ll} (q,5q-p) & \text{if}\;\;p/q \geq 2 \\ (p-q,2p-q) & \text{if}\;\;1/2 \leq p/q \leq 2 \\ (-p,2p-q) &  \text{if}\;\;p/q \leq 1/2\end{array}$  & 1,1 

 \\

\hline

$\CP^1 \times \CP^1$ &$F_2$  & $(-1,-2),(1,0)$ &  $\begin{array}{ll} (q,5q-p) & \text{if}\;\;p/q > 3 \\
(p-2q,p-q) & \text{if}\;\; 1 \leq p/q \leq 3 \\
(-p,p-q) & \text{if}\;\; p/q \leq 1
\end{array}$ & 1,2 
\\ 

\hline

$\CP^2 \#\ovl{\CP}^2$  & $F_2$ & $(1,0),(0,-1),(-1,-2)$ &
$\begin{array}{ll}
  (q,4q-p) & \text{if}\;\; p/q \geq2\\
  (p-q,p) & \text{if}\;\; 1 \leq p/q \leq 2\\
  (p-q,2p-q) & \text{if}\;\; 1/2 \leq p/q \leq 1\\
  (-p,2p-q) & \text{if}\;\; p/q \leq 1/2\\       
\end{array}$
& $1,1,1$  \\

 \hline

$\CP^2 \#^{\times 2}\ovl{\CP}^2$  & $F_1$ & $(1,0),(-1,-1),(0,-1)$ &
$\begin{array}{ll}
  (q,3q-p) & \text{if}\;\; p/q \geq 2\\
  (p-q,q) & \text{if}\;\; 1 < p/q \leq 2\\
  (p-q,2p-q) & \text{if}\;\; 1/2 \leq p/q \leq 1\\
  (-p,2p-q) & \text{if}\;\; p/q \leq 1/2\\       
\end{array}$

& $1,1,2$\\ 

\hline

$\CP^2 \#^{\times 3}\ovl{\CP}^2$  & $F_1$  & $(-1,-1),(1,0)$ &
$\begin{array}{ll}
  (q,3q-p) & \text{if}\;\; p/q \geq2\\
  (2p-3q,p-q) & \text{if}\;\; 1 \leq  p/q \leq 2\\
  (-p,p-q) & \text{if}\;\; p/q \leq 1\\
\end{array}$
& $2,3$\\ 

\hline

$\CP^2 \#^{\times 4}\ovl{\CP}^2$  & $F_2$ & $(1,0),(0,1)$ &
$\begin{array}{ll}
(p-2q,2p-3q) & \text{if}\;\; p/q \geq 2\\
(p-2q,2p-3q) & \text{if}\;\; 3/2 \leq p/q \leq 2\\
(q-p,2p-3q) & \text{if}\;\; 1 \leq p/q \leq 3/2\\
(q-p,2q-3p ) & \text{if}\;\; p/q \leq 1\\
\end{array}$
& $1,5$ \\ \hline
\end{tabular}

}
\end{table}

\section{From well-placed curves to scattering diagrams and back}\label{sec:from_wp_to_sd_and_back}

We now explain how to detect well-placed curves using scattering diagrams. Namely, given a uninodal Looijenga pair and a chosen toric model, Theorem~\ref{thm:wp_and_sd} states that there exists a $(p,q)$-well-placed curve if and only if a certain term in an associated scattering diagram is nonzero.
The main ingredient is the connection between scattering diagrams and certain Gromov--Witten invariants discovered in \cite{gross2010tropical}.

\subsection{Review of scattering diagrams}\label{subsec:scattering_review}

Let $\MM$ be a rank two lattice, put $\MM_\R := \MM \otimes_\Z \R$, and let $\C[\MM]$ denote the corresponding group algebra. 
As usual we will let $\NN$ denote the lattice dual to $\MM$.
For $\mm \in \MM$, we denote the corresponding monomial in $\C[\MM]$ by $z^\mm$.
Note that a choice of basis identifies $\MM$ with $\Z^2$ and $\C[\MM]$ with the algebra of Laurent polynomials $\C[x,x^{-1},y,y^{-1}]$, where $z^{(i,j)} = x^iy^j$.

In this paper, a \hl{wall} is a labeled ray $(\frakd,\ff)$, where:
\begin{itemize}
  \item $\frakd$ is an oriented ray in $\MM_\R$ with endpoint at the origin\footnote{We note that this is a specialization of a more general definition which does not require the ray endpoints to lie at the origin.} 
  \item $\ff \in \C[z^\mm]\ll t \rr \subset \C[\MM]\ll t \rr$ satisfies $\ff \equiv 1$ modulo $z^\mm t$, where $\mm \in \MM$ is the unique primitive element which spans the tangent space to $\frakd$ in the direction of its orientation.  
\end{itemize}

\NI We call a wall \hl{incoming} (resp. \hl{outgoing}) if its ray is oriented towards (resp. away from) the origin. 
A \hl{scattering diagram} $\calD$ in $\MM_\R$ is a multiset of walls in $\MM_\R$ such that for each $k \in \Z_{\geq 1}$ all but finitely many function labels are congruent to $1$ modulo $t^k$.
We note that the lattice $\MM \subset \MM_\R$ is an important implicit part of the data, since as we explain below its dual $\NN$ determines the wall crossing monodromy.

Let $(\frakd,\ff) \in \calD$ be a wall in $\MM_\R$, and consider a smooth path $[0,1] \ra \MM_\R$ which has a transverse intersection with $\frakd$ at some time $t_0 \in (0,1)$.
We have the associated \hl{wall-crossing monodromy} $\theta_{\ga,t_0}^{(\frakd,\ff),\calD} \in \aut_{\C\ll t \rr}(\C[\MM]\ll t \rr)$, which is defined on monomials by:
\begin{align}\label{eq:monod}
\theta_{\ga,t_0}^{(\frakd,\ff),\calD}: z^\mm \mapsto \ff^{\lan \nn,\mm\ran} z^\mm,
\end{align}
where $\nn \in \NN$ is the unique primitive element which vanishes on the tangent space to $\frakd$ and 
satisfies $\lan \nn,\ga'(0)\ran > 0$.

Now let $\ga: [0,1] \ra \MM_\R \setminus \{\vec{0}\}$ be a smooth immersion which intersects each wall of $\calD$ (or rather the corresponding ray) transversely.
Let $\theta_{\ga}^{\calD} \in \aut_{\C\ll t \rr}(\C[\MM]\ll t \rr)$ be given by composing the wall-crossing monodromies in order over every intersection point of $\ga$ with a wall of $\calD$.
Note that this is typically an infinite composition, but it is nevertheless well-defined since for any given $k \in \Z_{\geq 1}$ there are only finitely many nontrivial terms modulo $t^k$.
Two scattering diagrams $\calD,\calD'$ in $\MM_\R$ are \hl{equivalent} if $\theta_{\ga}^{\calD} = \theta_{\ga}^{\calD'}$ for any smooth immersion $\ga: [0,1] \ra \MM_\R \setminus \{\vec{0}\}$ which intersects the walls of both $\calD$ and $\calD'$ transversely.
Every scattering diagram $\calD$ is equivalent to a unique one $\calD_\min$ which is \hl{minimal}, i.e. no oriented rays are repeated and no label $\ff$ satisfies $\ff = 1$.
Indeed, we simply remove any wall $(\frakd,\ff)$ with $\ff=1$, and we replace any two walls $(\frakd_1,\ff_1),(\frakd_2,\ff_2)$ with the same oriented ray $\frakd_1 = \frakd_1$ with the wall $(\frakd_1,\ff_1 \ff_2)$. 

When $\ga$ is a loop, i.e. $\ga(0) = \ga(1)$, we refer to $\theta_{\ga}^{\calD}$ as the \hl{total monodromy} of $\calD$, noting that this is well-defined up conjugation (i.e. changing the starting point of $\ga$) and inversion (i.e. changing the orientation of $\ga$).
We will say that the scattering diagram $\calD$ is \hl{consistent} if 
the total monodromy is the identity $\1 \in \aut_{\C\ll t \rr}(\C[\MM]\ll t \rr)$.

Given any scattering diagram $\calD$ as above, Kontsevich--Soibelman \cite{Kontsevich2006} showed that there exists another scattering diagram $\ks(\calD)$, obtained by adding (typically infinitely many) outgoing walls to $\calD$, such that $\ks(\calD)$ is consistent. Moreover, $\ks(\calD)$ is unique up to equivalence, and hence its minimal representative $\ks(\calD)_\min$ is unique on the nose.
Concretely, $\ks(\calD)$ can be constructed algorithmically from $\calD$ by adding successive walls in order to kill the total monodromy order-by-order in $t$ (see e.g. \cite[Thm. 1.4]{gross2010tropical}).

\sss

Scattering diagrams of the following form, which we call \hl{basic},\footnote{A closely related notion of ``standard scattering diagrams'' appears in \cite[Def. 1.10]{gross2010tropical}, but due to a slightly different usage we use a different term here.} will play a distinguished role in the sequel. 

\begin{definition}\label{def:basic_sd}
 Given primitive vectors $\mm_1,\dots,\mm_J \in \MM: = \Z^2$ and positive integers $\ell_1,\dots,\ell_J \in \Z_{\geq 1}$ for some $J \in \Z_{\geq 2}$, let $\calD_{\mm_1,\dots,\mm_J}^{\ell_1,\dots,\ell_J}$ denote the scattering diagram in $\R^2$ with $J$ incoming walls given by
\begin{align*}
\calD_{\mm_1,\dots,\mm_J}^{\ell_1,\dots,\ell_J} := \left\{\big(\R_{\leq 0} \cdot \mm_1,(1+tz^{\mm_1})^{\ell_1}\big),\dots,\big(\R_{\leq 0} \cdot \mm_J,(1+tz^{\mm_J})^{\ell_J }\big)\right\}.
\end{align*}
\end{definition} 
\begin{example}
For each $\ell_1,\ell_2 \in \Z_{\geq 1}$, $\ks(\calD_{e_1,e_2}^{\ell_1,\ell_2})_\min$ consists of the incoming walls $\big(\R_{\leq 0} \cdot e_1,(1+tx)^{\ell_1}\big),\big(\R_{\leq 0} \cdot e_2,(1+ty)^{\ell_2}\big)$, the outgoing walls $\big(\R_{\geq 0} \cdot e_1,(1+tx)^{\ell_1}\big),\big(\R_{\geq 0} \cdot e_2,(1+ty)^{\ell_2}\big)$, and various additional outgoing walls of the form $\big(\R_{\geq 0} \cdot (a,b),\ff_{a,b}\big)$ for primitive $(a,b) \in \Z_{>0}^2$ and $\ff_{a,b} \in \C[z^{(a,b)}]\ll t\rr$.

For instance, for $\ell_1 = \ell_2 = 1$, $\ks(\calD_{e_1,e_2}^{1,1})_\min \setminus \calD_{e_1,e_2}^{1,1}$ consists of the single wall $\big(\R_{\geq 0} \cdot (1,1), 1+t^2xy\big)$.
For $\ell_1 = \ell_2 = 2$, the walls appearing in $\ks(\calD_{e_1,e_2}^{2,2})_\min \setminus \ks(\calD_{e_1,e_2}^{2,2})$ have slopes $(1,1),(k,k+1),(k+1,k)$ for all $k \in \Z_{\geq 1}$, with explicit function labels (see e.g. \cite[\S1.4]{gross2010quivers}).
For $\ell \geq 3$, the scattering diagram $\ks(\calD^{\ell,\ell}_{e_1,e_2})_\min$ is more complicated and involves a ``dense region'' where all rational slopes appear -- see \S\ref{subsec:wp_and_basic_sd} below.
\end{example}
\NI As we 
explain in \S\ref{subsec:change_lattice}, any basic scattering diagram with $J = 2$ can be reduced to one of the form $\calD_{e_1,e_2}^{\ell_1,\ell_2}$ for some $\ell_1,\ell_2 \in \Z_{\geq 1}$.

\subsection{Well-placed curves and Gromov--Witten theory}\label{subsec:wp_and_GW}

\NI We will associate to any uninodal Looijenga pair $(X,\calN)$ with a toric model $\tormod$ the scattering diagram 
\begin{align}\label{eq:D_tormod}
\calD_{\tormod} := \calD_{\mm_1,\dots,\mm_J}^{\ell_1,\dots,\ell_J},
\end{align}
 where $\mm_1,\dots,\mm_J \in \Z^2$ and $\ell_1,\dots,\ell_J \in \Z_{\geq 1}$ are the numerical data associated with $\tormod$ as in Notation~\ref{not:tormod_m_and_l} (here for concreteness we fix an identification $\MM \cong \Z^2$). 
We will say that the toric model $\tormod$ is \hl{strongly convex} if the rational polyhedral cone 
\begin{align*}
\cone(\mm_1,\dots,\mm_J)  := \left\{\sum\limits_{i=1}^Jc_i \mm_i\;|\; c_1,\dots,c_J \in \R_{\geq0} \right\} \subset \R^2
\end{align*}
is strongly convex, i.e. $\cone(\mm_1,\dots,\mm_J)$ does not contain any line (or, equivalently, if $c_1\mm_1 + \cdots + c_J\mm_J = \vec{0}$ for some $c_1,\dots,c_J \in \R_{\geq 0}$ then we must have $c_1 = \cdots = c_J = 0$).
Note that the rigid del Pezzo toric models in \S\ref{subsec:rigid_dP_tor_mods} are all strongly convex (for $J=2$ this is automatic as long as $\mm_1$ and $\mm_2$ are not colinear).

\begin{example}
Toric models for uninodal Looijenga pairs need not be strongly convex.
For instance, let $X$ be the blowup of $\CP^2$ at a smooth point on the nodal cubic $\calN_0 = \{x^3 + y^3 = xyz\}$, and let $\calN$ be the strict transform of $\calN_0$ in $X$. Since $(X,\calN)$ and $(\CP^2,\calN_0)$ coincide near the nodes, we can follow the same sequence of $3$ blowups described in \S\ref{subsec:rigid_dP_tor_mods} in order to obtain a toric model for $(X,\calN)$ with $J = 3$, $\mm_1 = (1,0), \mm_2 = (-1,-3), \mm_3 = (0,1)$ and $\ell_1 = \ell_2 = \ell_3 = 1$, where $\cone(\mm_1,\mm_2,\mm_3) = \R^2$ is evidently not strongly convex.
\end{example}
\begin{rmk}\label{rmk:scat_rays_lie_in_cone}
Using e.g. the connection between scattering diagrams and tropical curve counts discussed in \cite[Thm 2.4]{gross2010tropical}, one can show that the ray of every outgoing wall in $\ks(\calD_{\mm_1,\dots,\mm_J}^{\ell_1,\dots,\ell_J})_\min$ must lie in the cone $\cone(\mm_1,\dots,\mm_J)$. 
\end{rmk}

The following result shows that the existence of a $(p,q)$-well-placed curve in $X$ is equivalent to the nonvanishing of a certain term in the scattering diagram $\ks(\calD_{\tormod})_\min$ obtained from $\calD_{\tormod}$ by applying the Kontsevich--Soibelman algorithm discussed in \S\ref{subsec:scattering_review}.
For each primitive $\mm \in \Z^2$, let $\ff_\mm^\out \in \C[z^\mm]\ll t \rr$ be the function attached to the outgoing ray $\R_{\geq 0} \cdot \mm$ in $\ks(\calD_{\tormod})_\min$ (or $1$ if there is no such wall), let $\ff_\mm^\inn$ denote the same quantity but for the incoming ray $\R_{\leq 0} \cdot (-\mm)$, and put 
\begin{align}\label{eq:ff_mm}
\ff_\mm := \ff_\mm^\out \cdot \ff_\mm^\inn.
\end{align}
Note that for basic scattering diagrams as in Definition~\ref{def:basic_sd} we have $\ff_\mm^\inn = 1$ unless $\mm = -\mm_i$ for some $i \in \{1,\dots,J\}$.
For $\ka \in \Z_{\geq 1}$, let 
\begin{align}\label{eq:coef}
\coef_{\ks(\calD_\tormod)_\min}(z^{\ka \mm}) \in \C\ll t \rr
\end{align}
 denote the coefficient of $z^{\ka \mm}$ in $\log \ff_\mm$. 
Let $\wp_X: \Z^2_{\geq 1} \ra \Z^2$ be the function associated with the toric model $\tormod$ as in Definition~\ref{def:W}.
The following theorem is proved after Corollary~\ref{cor:curves_count_pos}. 

\begin{thm}\label{thm:wp_and_sd}
Let $(X,\calN)$ be a uninodal Looijenga pair with a toric model $\tormod$ and associated scattering diagram $\calD_\tormod$ as in \eqref{eq:D_tormod}.
For each coprime $p,q \in \Z_{\geq 1}$, if the scattering coefficient $\coef_{\ks(\calD_\tormod)_\min}(z^{\wp_X(p,q)})$ is nonzero, then there exists a rational algebraic curve in $X$ which is $(p,q)$-well-placed with respect to $\calN$. The converse is also true provided that $\tormod$ is strongly convex.
\end{thm}

  \begin{rmk}
In the case $\gcd(p,q) > 1$, it should in principle be possible to detect simple $(p,q)$-well-placed curves using the scattering diagram $\calD_\tormod$, by replacing $\coef_{\ks(\calD_\tormod)_\min}(z^{\wp_X(p,q)})$ with the corresponding BPS state counts defined in \cite[\S6]{gross2010tropical}, which subtract off multiple cover contributions.
These would typically correspond to curves in $X$ having a cusp with multiple Puiseux pairs (c.f. \cite[\S3.3]{mcduff2024singular}). 
  \end{rmk}

Before proving Theorem~\ref{thm:wp_and_sd}, we first recall the general connection between scattering diagrams and Gromov--Witten theory proved in \cite{gross2010tropical}, after introducing some necessary notation.
Suppose that for some $J \in \Z_{\geq 2}$ we are given:
\begin{itemize}
  \item pairwise distinct primitive integer vectors $\mm_1,\dots,\mm_J \in \Z^2$
  \item ordered partitions $\bP_1,\dots,\bP_J$,\footnote{By definition, an \hl{ordered partition} of length $\ell \in \Z_{\geq 1}$ is simply a tuple $\bP = (\rho^1,\dots,\rho^\ell)$ of nonnegative integers. We will denote the sum of its parts by $|\bP| = \rho^1 + \cdots + \rho^\ell$ and the length by $\len(\bP) = \ell$.} where $\bP_i = (\rho_i^1,\dots,\rho_i^{\ell_i})$  has length $\ell_i \in \Z_{\geq 1}$ for $i = 1,\dots,J$.
\end{itemize}
\NI 
Put $\mm_\out := |\bP_1| \mm_1 + \cdots + |\bP_J| \mm_J$. In the following, we will use the shorthand notation $\mmvec := (\mm_1,\dots,\mm_J)$ 
and $\bPvec := (\bP_1,\dots,\bP_J)$.
Following \cite{gross2010tropical}, we will define corresponding relative Gromov--Witten-type invariants $N_\mmvec[\bPvec] \in \Q$ which control the scattering coefficients in $\ks(\calD_{\mm_1,\dots,\mm_J}^{\ell_1,\dots,\ell_J})_\min$.

We will assume for ease of exposition that $\cone(-\mm_1,\dots,-\mm_J,\mm_\out) = \R^2$.\footnote{Note that the condition $\cone(-\mm_1,\dots,-\mm_J,\mm_\out) = \R^2$ ensures that $Y_\mmvec$ is a compact surface, and rules out e.g. $\mm_1 = (1,0), \mm_2 = (0,1), \mm_\out = (1,0)$. This condition can always be achieved by adding further rays to the fan, without essentially changing our curve counts of interest (at the cost of additional notation).}
 Let $Y_{\mmvec}$ denote the (typically singular) toric surface associated to the complete fan in $\R^2$ with ray generators $-\mm_1,\dots,-\mm_J,\mm_\out$, and let $Y^o_{\mmvec}$ be the result after removing all of the toric fixed points (i.e. $0$-dimensional orbits) from $Y_\mmvec$.
Here by slight abuse of notation $Y_\mmvec$ and $Y_\mmvec^o$ implicitly depend on $\bPvec$ (via $\mm_\out$). 
Let $\Ddiv_1,\dots,\Ddiv_J,\Ddiv_\out$ denote the toric divisors in $Y_\mmvec$ associated to $-\mm_1,\dots,-\mm_J,\mm_\out$ respectively, and let $\Ddiv_1^o,\dots,\Ddiv_J^o,\Ddiv_\out^o$ denote their respective restrictions to $Y_\mmvec^o$.
Note that in principle we allow the degenerate case when $\mm_\out$ is negatively proportional to some $\mm_i$, in which case $\Ddiv_\out = \Ddiv_i$ (this is ruled out if $\cone(\mm_1,\dots,\mm_J)$ is strongly convex).

For $i = 1,\dots,J$, let $x_i^1,\dots,x_i^{\ell_i}$ be pairwise distinct points in $\Ddiv_i^o$.
We denote the blowup of $Y_\mmvec$ at all of these points by $Y_\mmvec[\bPvec]$, with corresponding exceptional divisors $\E_i^j \subset Y_\mmvec[\bPvec]$ for $i=1,\dots,J$ and $j = 1,\dots,\ell_i$.
We will denote the strict transform of $\Ddiv_i \subset Y_\mmvec$ in $Y_\mmvec[\bPvec]$ by $\Ddiv_i[\bPvec]$, and (by slight abuse) we denote the strict transform of $\Ddiv_\out \subset Y_\mmvec$ in $Y_\mmvec[\bPvec]$ again by $\Ddiv_\out$.
Let $\Ddiv_1^o[\bPvec],\dots,\Ddiv_J^o[\bPvec],\Ddiv_\out^o$ denote the restrictions to $Y_\mmvec[\bPvec]^o$ of $\Ddiv_1[\bPvec],\dots,\Ddiv_J[\bPvec],\Ddiv_\out$ respectively.

Put $\ka := \gcd(\mm_\out) \in \Z_{\geq 1}$, and let $\be_{\bPvec} \in H_2(Y_\mmvec[\bPvec])$ be the homology class characterized by:
\begin{itemize}
  \item   $\be_\bPvec \cdot \Ddiv_\out = \ka$
  \item $\be_\bPvec \cdot \Ddiv_i[\bPvec] = 0$ for $i = 1,\dots,J$ (excluding $i=j$ if $\Ddiv_\out = \Ddiv_j[\bPvec]$)
  \item $\be_\bPvec \cdot \E_i^j = \rho_i^j$ for all $i = 1,\dots,J$ and $j \in 1,\dots,\ell_i$.
\end{itemize}
Let $\calM_{\be_\bPvec}(Y_\mmvec[\bPvec] / \Ddiv_\out)$ denote the moduli space of holomorphic maps $u: \CP^1 \ra  Y_\mmvec[\bPvec]$ such that
\begin{itemize}
  \item $u$ lies in homology class $\be_\bPvec$
  \item $u$ has full contact order $\ka$ with $\Ddiv_\out$ (at an unspecified point) at $\infty \in \CP^1$,
\end{itemize}
modulo biholomorphic reparametrizations of $\CP^1$ fixing $\infty$.
Note that any curve in $\calM_{\be_\bPvec}(Y_\mmvec[\bPvec] / \Ddiv_\out)$ necessarily has image contained in $Y^o_\mmvec[\bPvec]$ by positivity of intersections.

Now let $\ovl{\calM}_{\be_\bPvec}(Y_\mmvec[\bPvec] / \Ddiv_\out)$ denote the compactification of $\calM_{\be_\bPvec}(Y_\mmvec[\bPvec] / \Ddiv_\out)$ by relative stable maps (in the sense of \cite{li2001stable}), and let $\ovl{\calM}_{\be_\bPvec}(Y_\mmvec^o[\bPvec] / \Ddiv_\out^o) \subset \ovl{\calM}_{\be_\bPvec}(Y_\mmvec[\bPvec] / \Ddiv_\out)$ denote the open subspace of maps which avoid $Y_\mmvec[\bPvec] \setminus Y_\mmvec^o[\bPvec]$.
Strictly speaking, this is defined in \cite[\S5.1]{gross2010tropical} by first adding additional rays in order to desingularize $Y_\mmvec$, but we will suppress this from the notation.
According to \cite[Prop. 5.1]{gross2010tropical}, $\ovl{\calM}_{\be_\bPvec}(Y_\mmvec^o[\bPvec] / \Ddiv_\out^o)$ is compact and carries a natural virtual fundamental class, and we define Gromov--Witten-type invariants by:
\begin{align*}
N_\mmvec[\bPvec] := \# \ovl{\calM}_{\be_\bPvec}(Y_\mmvec^o[\bPvec] / \Ddiv_\out^o) := \int_{[\ovl{\calM}_{\be_\bPvec}(Y_\mmvec^o[\bPvec] / \Ddiv_\out^o)]^{\op{vir}}} 1 \in \Q.
\end{align*}

\begin{thm}[{\cite[Thm. 5.4]{gross2010tropical}, extended as in \cite[\S5.7]{gross2010tropical}}]\label{thm:GPS}

Fix $\mm_1,\dots,\mm_J \in \Z^2$ primitive and pairwise distinct and $\ell_1,\dots,\ell_J \in \Z_{\geq 1}$ for some $J \in \Z_{\geq 2}$, and
let $\calD_{\mm_1,\dots,\mm_J}^{\ell_1,\dots,\ell_J}$ be the associated basic scattering diagram in $\R^2$.
For each primitive $\mm \in \Z^2$, let $\ff_\mm^\out \in \C[z^\mm]\ll t \rr$ be the label attached to the outgoing ray $\R_{\geq 0} \cdot \mm$ in $\ks(\calD_{\mm_1,\dots,\mm_J}^{\ell_1,\dots,\ell_J})_\min$. 
Then we have:
\begin{align*}
\log \ff_\mm^\out = \sum\limits_{\ka = 1}^\infty \sum\limits_{|\bP_1|\mm_1 + \cdots + |\bP_J| \mm_J = \ka \mm} \ka N_\mmvec[\bPvec] t^{|\bP_1| + \cdots + |\bP_J|}  z^{\ka \mm},
\end{align*}
where the sum is over all ordered partitions $\bP_1,\dots,\bP_J$ of respective lengths $\ell_1,\dots,\ell_J$ such that $\sum\limits_{i=1}^J|\bP_i|\mm_i = \ka \mm$.
\begin{align*}
 \end{align*}
\end{thm}
\begin{rmk}
In particular, it follows that  $N_\mmvec[\bPvec]$ does not depend on the precise locations of the blowup points $\{x_i^j\}$ (this can also be checked directly by a compactness argument, c.f. \cite[\S5.2]{gross2010tropical}).
\end{rmk}

A typical element of $\ovl{\calM}_{\be_\bPvec}(Y_\mmvec^o[\bPvec] / \Ddiv_\out^o)$ consists of various curve components which are organized into a main level in $Y_\mmvec^o[\bPvec]$ (possibly vacuous) and some number (possibly zero) of ``neck'' levels in the $\CP^1$-bundle $\CP(\1_{\Ddiv^o_\out} \oplus N_{\Ddiv_\out^o}) \ra \Ddiv_\out^o$, where $N_{\Ddiv_\out^o}$ is the normal bundle of $\Ddiv_\out^o$ in $Y_\mmvec^o[\bPvec]$, subject to suitable matching, tangency, and stability conditions.\footnote{More formally, \cite{gross2010tropical} uses the language of {\em destabilizations}, while \cite{li2001stable} uses the language of {\em expanded degenerations}.}
In particular, there is a forgetful map which projects neck components down to $\Ddiv_\out^o$:
\begin{align*}
\for: \ovl{\calM}_{\be_\bPvec}(Y_\mmvec^o[\bPvec] / \Ddiv_\out^o) \ra \ovl{\calM}_{\be_\bPvec}(Y_\mmvec[\bPvec]),
\end{align*}
where the target is the usual moduli space of stable maps $\CP^1 \ra Y_\mmvec[\bPvec]$ in homology class $\be_\bPvec$.

Recall that we put $\ka := \gcd(\mm_\out)$, where $\mm_\out := \sum\limits_{i=1}^J |\bP_i| \mm_i$.
\begin{lemma}\label{lem:P_to_P_prime}
If $\ka = 1$ and $\ovl{\calM}_{\be_\bPvec}(Y^o_{\mmvec}[\bPvec] / \Ddiv_\out^o) \neq \nil$, then there exists a rational algebraic curve $\CP^1 \ra Y^o_\mmvec[\bPvec]$ which intersects $\Ddiv_\out^o$ transversely in one point and is otherwise disjoint from $\Ddiv_1^o[\bPvec],\dots,\Ddiv_J^o[\bPvec],\Ddiv_\out^o$.  
If we further assume that $\cone(\mm_1,\dots,\mm_J)$ is strongly convex, then we have 
\begin{align*}
\ovl{\calM}_{\be_\bPvec}(Y^o_{\mmvec}[\bPvec] / \Ddiv_\out^o) = \calM_{\be_\bPvec}(Y_{\mmvec}[\bPvec] / \Ddiv_\out).
\end{align*}
\end{lemma}
\begin{proof}
First note that we can assume that $\mm_\out$ is not negatively proportional to some $\mm_i$, since in that case we have $\Ddiv_\out = \Ddiv_i$, whence we can take our curve to be any of the exceptional divisors $\E_i^j$ for $j \in \{1,\dots,\ell_i\}$ (and this situation cannot occur if $\cone(\mm_1,\dots,\mm_J)$ is strongly convex).
Given $C \in \ovl{\calM}_{\be_\bPvec}(Y^o_{\mmvec}[\bPvec] / \Ddiv_\out^o)$, note that each component of $\for(C)  \in \ovl{\calM}_{\be_\bPvec}(Y_\mmvec[\bPvec])$ must have image distinct from each of $\Ddiv_1[\bPvec],\dots,\Ddiv_J[\bPvec],\Ddiv_\out$, and hence it must intersect each of these nonnegatively (c.f. \cite[\S4.2]{gross2010quivers}). 
By positivity of intersections and the definition of the homology class $\be_\bPvec \in H_2(Y_\mmvec[\bPvec])$, it follows that:
\begin{itemize}
  \item each component of $\for(C)$ has trivial intersection number with each of $\Ddiv_1[\bPvec],\dots,\Ddiv_J[\bPvec]$ 
  \item exactly one component $C_0$ of $\for(C)$ satisfies $C_0 \cdot \Ddiv_\out = 1$, and the remaining components of $\for(C)$ have trivial intersection number with $\Ddiv_\out$.
\end{itemize}
In particular, $C_0$ intersects $\Ddiv_\out$ transversely in exactly one point and is disjoint from each of $\Ddiv_1[\bPvec],\dots,\Ddiv_J[\bPvec]$.

Under the further assumption that $\cone(\mm_1,\dots,\mm_J)$ is strongly convex, we claim that $\for(C) = C_0$, which then implies $\ovl{\calM}_{\be_\bPvec}(Y^o_{\mmvec}[\bPvec] / \Ddiv_\out^o) = \calM_{\be_\bPvec}(Y_{\mmvec}[\bPvec] / \Ddiv_\out)$ by stability considerations. Since $\for(C)$ is a stable map, it suffices to show that any other component $C_1$ of $\for(C)$ would necessarily be constant.
To see this, note that for $\mm_\out' := \mm_\out/\ka$ with $\ka := \gcd(\mm_\out)$ and any $\nn \in \NN$, we have
\begin{align*}
\langle \mm_\out',\nn\rangle \Ddiv_\out -\sum_{i=1}^J \langle \mm_i,\nn\rangle \Ddiv_i
\end{align*}
vanishes as an element of the Chow group $A_{1}(Y_\mmvec)$ (see e.g. \cite[\S3.3]{fulton1993introduction}).
Letting $\uvl{C}_1$ denote the image of $C_1$ in $Y_\mmvec$, which necessarily satisfies $\uvl{C}_1 \cdot \Ddiv_\out = 0$, we then have
\begin{align*}
0 = (\Ddiv_\out \cdot \uvl{C}_1) \mm_\out'  = \sum_{i=1}^J  (\Ddiv_i \cdot \uvl{C}_1) \mm_i = \sum_{i=1}^J c_i\mm_i
\end{align*}
for $c_i := \Ddiv_i \cdot \uvl{C}_1 \in \Z_{\geq 0}$.
Strong convexity then implies $c_1 = \cdots = c_J = 0$, meaning that $\uvl{C}_1$ is constant, and hence $C_1$ is also constant.
\end{proof}

\begin{lemma}\label{lem:no_pos_dim_fams}
For any Looijenga pair $(X,\calN)$, there are no positive-dimensional families of rational algebraic curves in $X$ which intersect $\calN$ in a single point.
\end{lemma}
\begin{proof}
This follows exactly as in the proof of \cite[Lem. 1.1]{bousseau2021scattering}, which in turn is based on the argument that complex K3 surfaces are not uniruled (see e.g. \cite[\S4.1]{huybrechts2016lectures}).
Namely, let $\om$ be a holomorphic two-form on $X \setminus \calN$ with simple poles along $\calN$.
Given such a positive-dimensional family, we could find
a dominant rational map $F: \CP^1 \times S \dashrightarrow X$ for some Riemann surface $S$, such that $F^*\om$ is a holomorphic two-form on 
$(\CP^1 \setminus \{\infty\}) \times S$
 with simple poles along $\{\infty\} \times S$.
Let $\mathcal{X}$ be a nonvanishing holomorphic vector field defined on some open subset $U \subset S$, viewed as a vector field on $\CP^1 \times U$ which is trivial in the first factor.
Then, by contracting $F^*\om$ with $\mathcal{X}$ and restricting to $\CP^1 \times \{s\}$ for $s \in U$, we get a one-form on $\CP^1 \setminus \{\infty\}$ with a simple pole along $\{\infty\}$, which is a contradiction (the pole order must be at least $2$).
\end{proof}

In the strongly convex case, Lemma ~\ref{lem:P_to_P_prime} says that every curve in $\ovl{\calM}_{\be_\bPvec}(Y_\mmvec^o[\bPvec] / \Ddiv_\out^o)$ must be irreducible, while Lemma~\ref{lem:no_pos_dim_fams} further rules out positive dimensional families of irreducible curves.
 Since we are considering algebraic Gromov--Witten type invariants defined using integrable complex structures, these together have the following useful consequence. 

\begin{cor}\label{cor:curves_count_pos}
If $\ka = 1$ and $\cone(\mm_1,\dots,\mm_J)$ is strongly convex, then every curve in 
$\ovl{\calM}_{\be_\bPvec}(Y_\mmvec^o[\bPvec] / \Ddiv_\out^o)$
counts as a positive integer.
In particular, we have $N_\mmvec[\bPvec] \in \Z_{\geq 0}$, with $N_\mmvec[\bPvec] > 0$ if and only if $\calM_{\be_\bPvec}(Y_\mmvec[\bPvec] / \Ddiv_\out) \neq \nil$.
\end{cor}

As before, let $\tormod$ be a toric model for a uninodal Looijenga pair $(X,\calN)$ with associated data $\mmvec = (\mm_1,\dots,\mm_J)$ and $\ell_1,\dots,\ell_J \in \Z_{\geq 1}$ as in Notation~\ref{not:tormod_m_and_l}, and with toric Looijenga pair $(V^\tor,\Ddiv^\tor)$ and blowup set $\calS \subset \Ddiv^\tor$ as in Definition~\ref{def:tor_mod}. Given any ordered partitions $\bPvec = (\bP_1,\dots,\bP_J)$ with lengths $\ell_1,\dots,\ell_J$ and putting $\mm_\out := \sum\limits_{i=1}^J |\bP_i| \mm_i$,
note that we can identify $Y_{\mmvec}$ with $V^\tor_{+\mm_\out,\op{red}}$ and $Y_\mmvec[\bPvec]$ with $V^\tor_{+\mm_\out,\op{red}}[\calS]$ (recall Remark~\ref{rmk:Si_red_and_smth_vers}).

With these preliminaries, we are now ready to prove Theorem~\ref{thm:wp_and_sd}.

\begin{proof}[Proof of Theorem~\ref{thm:wp_and_sd}]

Suppose first that the scattering coefficient $\coef_{\ks(\calD_\tormod)_\min}(z^{\wp_X(p,q)})$ is nonzero. According to \eqref{eq:ff_mm}, we have either $\ff_{\wp_X(p,q)}^\out \neq 1$ or $\ff_{\wp_X(p,q)}^\inn \neq 1$ (or both).
In the former case, by Theorem~\ref{thm:GPS} we have 
$N_\mmvec[\bPvec] \neq 0$ for some ordered partitions $\bP_1,\dots,\bP_J$ of lengths $\ell_1,\dots,\ell_J$ such that $\sum\limits_{i=1}^J |\bP_i|\mm_i = \wp_X(p,q)$.
Thus $\ovl{\calM}_{\be_\bPvec}(Y_\mmvec^o[\bPvec] / \Ddiv_\out^o) \neq \nil$, and hence by Lemma~\ref{lem:P_to_P_prime} there exist a rational algebraic curve in $Y_\mmvec[\bPvec] \cong V^\tor_{+\wp_X(p,q),\op{red}}$ which intersects $\Ddiv_\out$ once transversely and is otherwise disjoint from $\Ddiv_1[\bPvec],\dots,\Ddiv_J[\bPvec],\Ddiv_\out$.
It then follows from the bijection Proposition~\ref{prop:fund_bij} that there exists a rational algebraic curve in $X$ which is $(p,q)$-well-placed with respect to $\calN$.

On the other hand, in the case $\ff_{\wp_X(p,q)}^\inn \neq 1$, we must have $\wp_X(p,q) = -\mm_i$ for some $i \in \{1,\dots,J\}$ and hence $V^\tor_{+\wp_X(p,q)}[\calS] = V^\tor[\calS]$ (by definition). Then there is an exceptional divisor of $V^\tor_{+\wp_X(p,q)}[\calS] \ra V_{+\wp_X(p,q)}^\tor$ which intersects $\Ddiv_\out$ transversely in one point and is otherwise disjoint from $\Ddiv^\tor_{+\wp_X(p,q)}[\calS]$, whence Proposition~\ref{prop:fund_bij} again produces a $(p,q)$-well-placed curve in $X$.

Now assume that $\tormod$ is strongly convex and that there exists a rational algebraic curve in $X$ which is $(p,q)$-well-placed with respect to $\calN$, where $\gcd(p,q) = 1$.  Then $\gcd(\wp_X(p,q)) = 1$ by Proposition~\ref{prop:fund_bij}, and,
assuming that $\wp_X(p,q)$ is not equal to $-\mm_i$ for some $i \in \{1,\dots,J\}$, 
it follows by Proposition~\ref{prop:fund_bij} that we have $\calM_{\be_\bPvec}(Y_\mmvec[\bPvec] / \Ddiv_\out) \neq \nil$ for some ordered partitions $\bPvec = (\bP_1,\dots,\bP_J)$ of lengths $\ell_1,\dots,\ell_J$ such that $\sum_{i=1}^J |\bP_i| \mm_i = \wp_X(p,q)$.
By Theorem~\ref{thm:GPS} we then have 
$$
\coef_{\ks(\calD_\tormod)_\min}(z^{\wp_X(p,q)}) = \sum\limits_{|\bP_1'|\mm_1 + \cdots + |\bP_J'|\mm_J = \wp_X(p,q)}N_\mmvec[\bPvec'] t^{|\bP_1'| + \cdots + |\bP_J'|}.
$$
Since by Corollary~\ref{cor:curves_count_pos} the terms $N_{\mmvec'}[\bPvec']$ are all nonnegative, we have $N_\mmvec[\bPvec] > 0$, whence $\coef_{\ks(\calD_\tormod)_\min}(z^{\wp_X(p,q)}) \neq 0$.

Finally, if $\wp_X(p,q) = -\mm_i$ for some $i \in \{1,\dots,J\}$, then by strong convexity there are no outgoing rays $\R_{\geq 0} \cdot \wp_X(p,q)$ appearing in $\ks(\calD_\tormod)_\min$ (c.f. Remark~\ref{rmk:scat_rays_lie_in_cone}), and hence $\ff_{\wp_X(p,q)} = \ff_{-\mm_i}^\inn = (1+tz^{\mm_i})^{\ell_i}$, so 
$\coef_{\ks(\calD_\tormod)_\min}(z^{\wp_X(p,q)}) = \ell_i t \neq 0$.
\end{proof}

\sss

Observe that if $\tormod$ is a strongly convex toric model for a uninodal Looijenga pair $(X,\calN)$ with data $\mm_1,\dots,\mm_J \in \Z^2$ and $\ell_1,\dots,\ell_J \in \Z_{\geq 1}$, then for any given $\mm \in \Z^2$ there are only finitely many ordered partition tuples $\bPvec = (\bP_1,\dots,\bP_J)$ such that $|\bP_1|\mm_1 + \cdots + |\bP_J|\mm_J = \mm$.
It follows that $\coef_{\ks(\calD_\tormod)_\min}(z^{\mm})$ is a polynomial in $t$ (i.e. it lies in $\C[t]$), and in particular has a well-defined $t=1$ specialization $\coef_{\ks(\calD_\tormod)_\min}(z^{\mm})|_{t=1} \in \C$.
By the results of this section, for coprime $p,q \in \Z_{\geq 1}$, the quantity
\begin{align}\label{eq:p_q_GW_count_def}
\wpcount_{X,\calN}(p,q)& := 
\sum_{
\sum|\bP_i|\mm_i = \wp_X(p,q)} N_\mmvec[\bPvec] = 
\coef_{\ks(\calD_\tormod)_\min}(z^{\wp_X(p,q)})|_{t=1}
\end{align}
can be interpreted as the algebraic count of rational algebraic curves in $X$ which are $(p,q)$-well-placed with respect to $\calN$.

Although in this paper we are primarily concerned with understanding when $\wpcount_{X,\calN}(p,q)$ is nonzero, it is also very natural to study the counts themselves. 
For example, one can show that $\wpcount_{X,\calN}(p,q) = 1$ whenever $X$ is a rigid del Pezzo surface and $p/q$ is  the $x$-value of an outer corner point of the infinite staircase $c_X|_{[1,a_\acc^X]}$, and we have $\wpcount_{\CP^2,\calN_0}(p,q) = 3$ for the ghost staircase points (i.e. $\tfrac{p}{q} = \tfrac{8}{1},\tfrac{55}{8},\tfrac{377}{55}$ etc, c.f. \cite[\S7.1]{mcduff2024singular}).
In the case $q = 1$, computer experiments suggest the following conjectural formula for all $d \in \Z_{\geq 1}$:
\begin{align*}
\wpcount_{\CP^2,\calN_0}(3d-1,1) = \tfrac{2(4d-3)!}{d!(3d-1)!}.
\end{align*}
Note that any given count $\wpcount_{X,\calN}(p,q)$ can easily be computed algorithmically (see e.g. \cite{graefnitz_scattering}), and, when $X$ is a rigid del Pezzo surface, each such count agrees with infinitely many others using the symmetries discussed in \S\ref{sec:symmetries} (or their scattering diagram counterparts as in \cite[\S5]{gross2010quivers}). 

\begin{rmk}\label{rmk:refined_SD}
The notion of scattering diagrams considered in this paper has a natural generalization where we work over
 $\C[\MM]\ll t_1,\dots,t_J\rr$ rather than $\C[\MM]\ll t\rr$,
 and in the scattering diagram $\calD_\tormod = \calD_{\mm_1,\dots,\mm_J}^{\ell_1,\dots,\ell_J}$ associated to a toric model $\tormod$ as in \eqref{eq:D_tormod} we instead take the initial rays to be of the form $\bigl(\R_{\leq 0} \cdot \mm_i,(1+t_iz^{\mm_i})^{\ell_i}\bigr)$ for $i=1,\dots,J$.
Since \cite[Thm. 5.4]{gross2010tropical} holds in this setting, we also get an extension of Theorem~\ref{thm:wp_and_sd}, where essentially for a ray appearing with coefficient $t_1^{k_1}\cdots t_J^{k_J}$, the total coefficient $k_1+\cdots+k_J$ gives what was previously the coefficient of $t$, whereas the particular decomposition $(k_1,\dots,k_m)$ encodes the homology class of a $(p,q)$-well-placed curve in our uninodal Looijenga pair $(X,\calN)$.

Over these refined coefficients, even the basic scattering diagrams discussed in \S\ref{sec:std_sd} are substantially more complicated when $J \geq 3$, and are not (to our knowledge) well-understood. Under a homology-level refinement of the fundamental bijection, this is akin to classifying the homology classes of well-placed curves in $(X,\calN)$.
Note that the extra homology data in particular encodes the count of singularities of a well-placed curve away from its distinguished cusp (c.f. Remark~\ref{rmk:adj_not_compl_obs}).

For example, the classification of index zero unicuspidal rational algebraic curves in the first Hirzebruch surface discussed in \cite[\S6.5]{mcduff2024singular} (based on \cite{magillmcd2021,magill2022staircase}) exhibits subtle Cantour set behavior, which puts a lower bound of the complexity of the corresponding scattering diagram with refined coefficients.
A fuller understanding of this scattering diagram appears to play an important role in various generalizations of Theorems \ref{thmlet:main_SEEP} and \ref{thmlet:SEEP_dPs} (e.g. the general case of Problem~\ref{prob:SEEP}), which we plan to address in a followup paper.
\end{rmk}

\section{Basic scattering diagrams}\label{sec:std_sd}

Using Theorem~\ref{thm:wp_and_sd}, we have now reduced the existence of $(p,q)$-well-placed curves in a uninodal Looijenga pair to the nonvanishing of certain terms in an associated scattering diagram, so it remains to understand when these scattering terms are nonzero. 
We first discuss in \S\ref{subsec:change_lattice} the change of lattice trick from \cite[\S C.3]{GHKK2018}, which reduces the study of arbitary basic scattering diagrams with $J=2$ incoming rays to simplified diagrams of the form $\calD_{e_1,e_2}^{\ell_1,\ell_2}$ for some $\ell_1,\ell_2 \in \Z_{\geq 1}$. Then, in \S\ref{subsec:scat_pos} we review some recent advances in \cite{gross2010quivers,GHKK2018,grafnitz2023scattering} which imply scattering positivity for these simplified diagrams.
Finally, in \S\ref{subsec:wp_and_basic_sd} we combine the preceding results in order to conclude the proofs of Theorems ~\ref{thmlet:sesqui_plane_curves} and ~\ref{thmlet:curves_in_X}.

\subsection{Changes of lattice}\label{subsec:change_lattice}

Let $\MM$ be a rank two lattice, and let $\MM' \subset \MM$ be a rank two sublattice, with finite index denoted by $\ind(\MM' \subset \MM) \in \Z_{\geq 1}$.
Note that there is an inclusion-induced linear isomorphism $\MM'_\R \xrightarrow{\cong} \MM_\R$, and the dual lattice $\NN := \hom_\Z(\MM,\Z)$ is a sublattice of $\NN' := \hom_\Z(\MM',\Z)$.
For each nonzero $\mm' \in \MM'$, put
\begin{align}\label{eq:nu_def}
\nu(\mm') := \ind\left( \ann_{\NN}(\mm') \subset \ann_{\NN'}(\mm') \right),
\end{align}
where $\ann_\NN(\mm') := \{\nn \in \NN\;|\; \lan \nn,\mm'\ran = 0\}$ (resp. $\ann_{\NN'}(\mm')$) denotes the annihilator of $\mm'$ in $\NN$ (resp. $\NN'$).
Noting that $\nu(\mm')$ depends only on the ray $\frakd = \R_{\geq 0} \cdot \mm'$ spanned by $\mm'$, we will sometimes also denote $\nu(\mm')$ by $\nu(\frakd)$.

\begin{example}\label{ex:m_1_m_2_nu}
Let $\MM'$ be the sublattice of $\MM = \Z^2$ generated by $\mm_1 = (1,0)$ and $\mm_2 = (-1,-3)$.
Note that we have the natural identifications $\NN \cong \Z^2$ and
\begin{align*}
\NN' \cong \{(i,j) \in \R^2 \; | \; \lan (i,j),\mm_1\ran, \lan(i,j),\mm_2\ran \in \Z\} = \Z \times \tfrac{1}{3}\Z.
\end{align*}
We have $\ann_{\NN}(\mm_1) = \Z\lan (0,1)\ran$ and 
$\ann_{\NN'}(\mm_1) = \Z\lan (0,\tfrac{1}{3})\ran$, and thus $\nu(\mm_1) = 3$.
Similarly, we have $\ann_{\NN}(\mm_2) = \Z\lan (3,-1)\ran$ and $\ann_{\NN'}(\mm_2) = \Z\lan (1,-\tfrac{1}{3})\ran$, and thus $\nu(\mm_2) = 3$.
On the other hand, we have 
$\nu(\mm_1 + \mm_2) = \nu(0,-3) = \nu(0,-1) = 1$.
\end{example}

For any rank two lattice $\MM$, $\nu \in \Z_{\geq 1}$, and $\ff \in \C[\MM]\ll t\rr$ with $\ff \equiv 1$ modulo $t$, let $\ff^{1/\nu} \in \C[\MM]\ll t \rr$ denote the unique $\nu$th root of $\ff$ satisfying $\ff^{1/\nu} \equiv 1$ modulo $t$.
The following simple observation is the basis of the ``change of lattice trick'' from \cite[\S C.3]{GHKK2018}. 

\begin{lemma}\label{lem:subscatter}
Let $\MM'$ be a rank two sublattice of a rank two lattice $\MM$, let $\calD$ be a scattering diagram in $\MM'_\R$ which is consistent, and let $\calDroot$ be the scattering diagram in $\MM_\R$ defined by
\begin{align*}
\calDroot := \{(\frakd,\ff^{1/\nu(\frakd)}) \;|\; (\frakd,\ff) \in \calD\}.
\end{align*}
Then $\calDroot$ is also consistent.
\end{lemma}

\NI In particular, since $\ks(\calD)$ is consistent for any scattering diagram $\calD$ in $\MM_\R'$, we have that $\ks(\calD)_\root$ is a consistent scattering diagram in $\MM_\R$.
Then $\ks(\calD)_\root$ and $\ks(\calD_\root)$ must be equivalent since they are both consistent and have the same incoming walls, and so we have
\begin{align}\label{eq:min_and_root}
\ks(\calDroot)_\min = \big(\ks(\calD)_\root \big)_\min = \big(\ks(\calD)_\min \big)_\root.
\end{align}

\begin{proof}[Proof of Lemma~\ref{lem:subscatter}]
Suppose that $\ga: [0,1] \ra \MM_\R'$ is a smooth loop which intersects each wall of $\calD$ transversely.
Let $(\frakd,\ff)$ be a wall of $\calD$ such that $\ga$ intersects $\frakd$ at some $t_0 \in (0,1)$.
Let $\nn \in \NN$ (resp. $\nn' \in \NN'$) be the unique primitive element which vanishes on $\frakd$ and pairs positively with $\ga'(0)$, so that we have
$\nn = \nu(\frakd) \nn'$.
Then for any $\mm' \in \MM'$ we have
\begin{align*}
\theta_{\ga,t_0}^{(\frakd,\ff^{1/\nu(\frakd)}),\calDroot}(z^{\mm'}) &= 
(\ff^{1/\nu(\frakd)})^{\lan \nn,\mm'\ran} z^{\mm'}
\\&  = (\ff^{1/\nu(\frakd)})^{\nu(\frakd)\lan \nn',\mm'\ran}z^{\mm'} \\&= \ff^{\lan \nn',\mm'\ran}z^{\mm'}  \\&= \theta_{\ga,t_0}^{(\frakd,\ff),\calD}(z^{\mm'}).
\end{align*}
It follows that $\theta_{\ga}^{\calDroot}(z^{\mm'}) = \theta_{\ga}^{\calD}(z^{\mm'}) = z^{\mm'}$ for any $\mm' \in \MM'$.

To conclude that $\theta_{\ga}^{\calDroot} = \1 \in \aut_{\C\ll t \rr}(\C[\MM]\ll t \rr)$, note that for any $\mm \in \MM$ we have $K\mm \in \MM'$ for some $K \in \Z_{\geq 1}$, so
\begin{align*}
(\theta_{\ga}^{\calDroot}(z^{\mm}))^K = \theta_{\ga}^{\calDroot}(z^{K\mm}) = z^{K\mm} = (z^{\mm})^K,
\end{align*}
and hence $\theta_{\ga}^{\calDroot}(z^{\mm}) = z^{\mm}$ since $\theta_{\ga}^{\calDroot} = \1$ modulo $t$.
\end{proof}

There is also a straightforward notion of isomorphism of scattering diagrams which plays well with the Kontsevich--Soibelman algorithm.
Given a lattice isomorphism $\phi: \MM_1 \xrightarrow{\cong} \MM_2$, let $\phi_\R: (\MM_1)_\R \ra (\MM_2)_\R$ denote the induced isomorphism of real vector spaces.
For an oriented ray $\frakd \subset (\MM_1)_\R$, we endow the corresponding ray $\phi_\R(\frakd) \subset (\MM_2)_\R$ with its induced orientation.
For $\mm \in \MM_1$ nonzero and $\ff \in \C[z^\mm]\ll t\rr$, we obtain an element $\phi_*(\ff) \in \C[z^{\phi(\mm)}]\ll t \rr$ by replacing each instance of $z^\mm$ with $z^{\phi(\mm)}$.
The following is more or less immediate from the definitions:

\begin{lemma}\label{lem:scatter_coord_ch}
Let $\phi: \MM_1 \xrightarrow{\cong} \MM_2$ be an isomorphism of rank two lattices, let $\calD_1$ be a scattering diagram in $(\MM_1)_\R$ which is consistent, and let $\phi_*(\calD_1)$ be the scattering diagram in $(\MM_2)_\R$ defined by
\begin{align*}
\phi_*(\calD_1) := \{(\phi_\R(\frakd),\phi_*(\ff)\;|\; (\frakd,\ff) \in \calD_1 \}.
\end{align*}
Then $\phi_*(\calD_1)$ is also consistent.
\end{lemma}

\begin{example}\label{ex:two_lemmas_main_ex}
Consider the lattices $\MM = \Z^2$ and $\MM' = \langle \mm_1,\mm_2\rangle = \Z \times 3\Z$ for $\mm_1 = (1,0)$ and $\mm_2 = (-1,-3)$,
and let $\calD$ denote $\calD_{\mm_1,\mm_2}^{3,3}$, but viewed as a scattering diagram in $\MM'_\R$, i.e. with respect to the lattice $\MM'$.
Then $\calDroot$ is the basic scattering diagram $\calD_{\mm_1,\mm_2}^{1,1}$ in $\MM_\R = \R^2$, and by \eqref{eq:min_and_root} we have
$\ks(\calD_{\mm_1,\mm_2}^{1,1})_\min  = \big(\ks(\calD)_\min\big)_\root$.
Thus in order to compute $\ks(\calD_{\mm_1,\mm_2}^{1,1})_\min$ it essentially suffices to compute $\ks(\calD)_\min$.
In turn, using the isomorphism $\phi: \MM' \ra \Z^2$ sending $\mm_i$ to $e_i$ for $i=1,2$, Lemma~\ref{lem:scatter_coord_ch} reduces $\ks(\calD)_\min$ to the basic scattering diagram $\ks(\calD_{\phi(\mm_1),\phi(\mm_2)}^{3,3})_\min = \ks(\calD_{e_1,e_2}^{3,3})_\min$ in $\R^2$.
\end{example}

Generalizing Example~\ref{ex:two_lemmas_main_ex}, we will typically apply the above lemmas as follows.
Put $\MM := \Z^2$, and let $\calD_{\mm_1,\dots,\mm_J}^{\ell_1,\dots,\ell_J}$ denote the basic scattering diagram in $\MM_\R = \R^2$ specified by some primitive vectors $\mm_1,\dots,\mm_J \in \Z^2$ and positive integers $\ell_1,\dots,\ell_J \in \Z_{\geq 1}$ as in Definition~\ref{def:basic_sd}.
Let $\MM' \subset \MM$ be a sublattice containing $\mm_1,\dots,\mm_J$, and let $\calD$ be the scattering diagram in $\MM'_\R$ given by
\begin{align*}
\calD = \{(\R_{\leq 0} \cdot \mm_i,(1+tz^{\mm_i})^{\nu(\mm_i)\ell_i}) \; | \; i = 1,\dots,J\}.
\end{align*}
According to Lemma~\ref{lem:subscatter}, the scattering diagram $\ks(\calD)_\min$ straightforwardly determines $\ks(\calD_{\mm_1,\dots,\mm_J}^{\ell_1,\dots,\ell_J})_\min$ and vice versa.
Moreover, after choosing a lattice isomorphism $\phi: \MM' \xrightarrow{\cong} \Z^2$, we can identify $\calD$ via Lemma~\ref{lem:scatter_coord_ch} with another basic scattering diagram
$\calD_{\phi(\mm_1),\dots,\phi(\mm_J)}^{\nu(\mm_1)\ell_1,\dots,\nu(\mm_J)\ell_J}$.

In the particular case $J=2$, we can take  $\MM'$ to be the sublattice  $\lan \mm_1,\mm_2 \ran$ of $\MM = \Z^2$, and $\phi: \MM' \ra \Z^2$ to be the lattice isomorphism with $\phi(\mm_i) = e_i$ for $i=1,2$.
Note that each outgoing ray appearing in $\ks(\calD_{\mm_1,\mm_2}^{\ell_1,\ell_2})_\min$ is necessarily of the form $\R_{\geq 0} \cdot \mm$ for some $\mm = a\mm_1 + b\mm_2 \in \Z^2$ with $(a,b) \in \Z_{\geq 0}^2$. 
Let $\ff_{\mm_1,\mm_2}^{\ell_1,\ell_2}(\mm) \in \C[z^\mm]\ll t \rr$ denote the label attached to the ray $\R_{\geq 0}\cdot \mm$ in $\ks(\calD_{\mm_1,\mm_2}^{\ell_1,\ell_2})_\min$. 
Combining Lemmas \ref{lem:subscatter} and \ref{lem:scatter_coord_ch} then gives:
\begin{cor}\label{cor:f_change_formula}
For any $\ell_1,\ell_2 \in \Z_{\geq 1}$, primitive noncolinear $\mm_1,\mm_2 \in \Z^2$, and $a,b \in \Z_{\geq 0}$, we have
\begin{align}\label{eq:comp_with_std_sd}
\ff_{\mm_1,\mm_2}^{\ell_1,\ell_2}(a\mm_1 + b\mm_2) = \left(\phi_*^{-1}(\ff_{e_1,e_2}^{\nu(\mm_1)\ell_1,\nu(\mm_2)\ell_2}(a,b))\right)^{1/\nu(a\mm_1 + b\mm_2)},
\end{align}
where $\phi_*^{-1}(\ff_{e_1,e_2}^{\nu(\mm_1)\ell_1,\nu(\mm_2)\ell_2}(a,b))$ is given by replacing each instance of $z^{(a,b)}$ in $\ff_{e_1,e_2}^{\nu_1\ell_1,\nu_2\ell_2}(a,b)$ with $z^{a\mm_1+b\mm_2}$.
\end{cor}

Because $\coef_{\ks(\calD_{\mm_1,\mm_2}^{\ell_1,\ell_2})_\min}(z^{\mm})$ is defined in \eqref{eq:coef} as a coefficient in the function $\log \ff_\mm$, we may conclude:

\begin{cor}\label{cor:COL_calD}
In the context of Corollary~\ref{cor:f_change_formula}, we have 
\begin{align*}
\coef_{\ks(\calD_{\mm_1,\mm_2}^{\ell_1,\ell_2})_\min}(z^{\mm}) = 
\begin{cases}
\tfrac{1}{\nu(\mm)}\; \coef_{\ks(\calD_{e_1,e_2}^{\nu(\mm_1)\ell_1,\nu(\mm_2)\ell_2})_\min}(z^{(a,b)}) & \text{if}\;\;\mm = a\mm_1 + b\mm_2 \\   
0 & \text{if}\;\;\mm \in \MM' \setminus \MM.
\end{cases}
\end{align*}  
\end{cor}

As a consequence, given a toric model of a Looijenga pair $(X,\calN)$ with $J=2$ incoming walls, we can read off counts $\wpcount_{X,\calN}(p,q)$ of well-placed curves as in \eqref{eq:p_q_GW_count_def} by applying the Kontsevich--Soibelman algorithm to  particularly simple basic scattering diagrams of the form studied in e.g. \cite{gross2010quivers,grafnitz2023scattering}.

\begin{cor}\label{cor:N_logcoef_form}
Suppose that $(X,\calN)$ is a uninodal Looijenga pair which has a toric model $\tormod$ with $J=2$, with data $\mm_1,\mm_2,\ell_1,\ell_2$ as in Notation~\ref{not:tormod_m_and_l}.  
Then for any coprime $p,q \in \Z_{\geq 1}$ we have
\begin{align*}
\wpcount_{X,\calN}(p,q) &= 
\begin{cases}
  \tfrac{1}{\nu(\wp_X(p,q))}\,\coef_{\ks(\calD_{e_1,e_2}^{\nu(\mm_1)\ell_1,\nu(\mm_2)\ell_2})_\min}(z^{\phi(\wp_{X}(p,q)}))|_{t=1} & \text{if}\;\; \wp_{X}(p,q) \in \lan \mm_1,\mm_2\ran\\
  0 &\text{if}\;\; \wp_{X}(p,q) \notin \lan \mm_1,\mm_2\ran,
\end{cases} 
\end{align*}
where $\phi$ is the linear map $\R^2 \ra \R^2$ with $\phi(\mm_i) = e_i$ for $i=1,2$.
\end{cor}

\begin{example}\label{ex:nu_CP2}
Recall that $(\CP^2,\calN_0)$ has a toric model $\tormod_{\CP^2}$ with $\mm_1 = (1,0), \mm_2 = (-1,-3)$ and $\ell_1 = \ell_2 = 1$, where $\calN_0 =  \{x^3 + y^3 = xyz\}$ is our standard nodal cubic.
 Continuing Examples ~\ref{ex:m_1_m_2_nu} and \ref{ex:two_lemmas_main_ex}, observe that for any primitive $\mm = (i,j) \in \Z^2$ we have 
\begin{align*}
\nu(\mm) = 
\begin{cases}
  3 &  \text{if}\;\;j \equiv 0 \;\text{mod}\; 3\\ 
  1 & \text{otherwise},
\end{cases}
\end{align*}
and in fact $\nu(\wp_{\CP^2}(p,q)) = 3$ for any coprime $p,q \in \Z_{\geq 1}$ with $p+q$ divisible by $3$.
Together with Corollary~\ref{cor:N_logcoef_form}, we then have
\begin{align}
N_{\CP^2,\calN_0}(p,q) = \tfrac{1}{3} \,\coef_{\ks(\calD_{e_1,e_2}^{3,3})_\min}(z^{\phi(\wp_{\CP^2}(p,q))})|_{t=1},
\end{align}
where $\phi$ is represented by the matrix 
$\tfrac{1}{3}\begin{psmallmatrix}
3 & -1 \\ 0 &-1
\end{psmallmatrix}
= \begin{psmallmatrix}
1 & -1 \\ 0 & -3
\end{psmallmatrix}^{-1}$. 
\end{example}

\begin{rmk} 
In principle we could extend the count $N_{X,\calN}(p,q)$ of $(p,q)$-well-placed curves to the case with $\gcd(p,q) > 1$ using \eqref{eq:p_q_GW_count_def}, although this generally involves multiple covers and possibly more degenerate configurations.
A noteworthy special case is when $(p,q) = (\ka,0)$ (or equivalently $(0,\ka)$) for some $\ka \in \Z_{\geq 1}$, interpreted as in Convention~\ref{conv:pq=0} as the count of rational algebraic curves in $X$ which intersect $\calN$ in one nonsingular point with contact order $\ka$.
For instance, in the case $X = \CP^2$, we have $\wp_{\CP^2}(3\ka,0) = (0,-3\ka)$, which is sent to $(\ka,\ka)$ by the map $\phi$ from Example~\ref{ex:nu_CP2}.
According to a suitable extension of Corollary~\ref{cor:N_logcoef_form}, $N_{\CP^2,\calN_0}(3\ka,0)$ can be read off from the label $\ff_{e_1,e_2}^{3,3}(1,1)$ of the ray $\R_{\geq 0} \cdot (1,1)$ in $\ks(\calD_{e_1,e_2}^{3,3})_\min$.
As it happens, an explicit formula for $\ff_{e_1,e_2}^{\ell,\ell}(1,1)$ for any $\ell \in \Z_{\geq 1}$ was conjectured by Gross and Kontsevich and proved by Reineke in \cite{reineke2011cohomology}.
Incidentally, the analogous count in $(\CP^2,E)$ with $E$ a nonsingular elliptic curve has received much recent attention in the form of Takahashi's conjecture (see e.g. \cite{bousseau2021scattering}).
\end{rmk}

\subsection{Scattering positivity results}\label{subsec:scat_pos}

Basic scattering diagrams of the form $\ks(\calD_{e_1,e_2}^{\ell_1,\ell_2})_\min$ for $\ell_1,\ell_2 \in \Z_{\geq 1}$ were discussed in detail in \cite{gross2010quivers}, and studied empirically based on computer calculations in e.g. \cite[Ex. 1.6]{gross2010tropical} and \cite[Ex. 1.15]{GHKK2018}.
In particular, \cite[\S4]{gross2010quivers} gives a complete conjectural picture for the \hl{scattering pattern} of these scattering diagrams, i.e. the set of all rays with nontrivial function labels. 
Note that, apart from the incoming rays $\R_{\leq 0}\cdot (1,0)$ and $\R_{\leq 0} \cdot (0,1)$ and the outgoing rays $\R_{\geq 0} \cdot (1,0)$ and $\R_{\geq 0} \cdot (0,1)$, all other rays of $\ks(\calD_{e_1,e_2}^{\ell_1,\ell_2})_\min$ are outgoing with positive rational slope.

We will restrict to the case $\ell_1\ell_2 > 4$, since the remaining cases are much simpler and not directly relevant for us.
Let 
\begin{align}\label{eq:xipm}
\xi^{\ell_1,\ell_2}_\pm := \tfrac{\ell_2}{2}\left(1 \pm \sqrt{1 - \frac{4}{\ell_1\ell_2}} \right) \in \R_{>0}
\end{align}
 denote the roots of the polynomial $R_{\ell_1,\ell_2}(t) = \tfrac{1}{\ell_2}t^2 - t + \tfrac{1}{\ell_1}$, and define the \hl{dense region} of $\ks(\calD_{e_1,e_2}^{\ell_1,\ell_2})_\min$ in $\R^2$ to be the set of all $(a,b) \in \R_{>0}^2$ satisfying 
$\xi_-^{\ell_1,\ell_2} < b/a < \xi_+^{\ell_1,\ell_2}$.
More generally, for any primitive noncolinear $\mm_1,\mm_2 \in \Z^2$, we define the dense region of $\ks(\calD_{\mm_1,\mm_2}^{\ell_1,\ell_2})_\min$ to be the preimage of the dense region of $\ks(\calD_{e_1,e_2}^{\nu(\mm_1)\ell_1,\nu(\mm_2)\ell_2})_\min$ under the map $\phi: \R^2 \ra \R^2$ sending $\mm_i$ to $e_i$ for $i=1,2$.

According to \cite[Thm. 5]{gross2010quivers}, those primitive $(a,b) \in \Z_{\geq 1}^2$ such that the ray $\R_{\geq 0} \cdot (a,b)$ appears in the scattering pattern of $\ks(\calD_{e_1,e_2}^{\ell_1,\ell_2})_\min$ and lies {\em outside} of the dense region are precisely of the form $\Tsym_2(1,0),\Tsym_1(\Tsym_2(1,0)),\Tsym_2(\Tsym_1(\Tsym_2(1,0))),\dots$ and $\Tsym_1(0,1),\Tsym_2(\Tsym_1(0,1)),\Tsym_1(\Tsym_2(\Tsym_1(0,1))),\dots$, where $\Tsym_i := \Tsym_i^{\ell_1,\ell_2}: \Z^2 \ra \Z^2$ for $i=1,2$ are involutive symmetries of the scattering diagram $\ks(\calD_{e_1,e_2}^{\ell_1,\ell_2})_\min$, given by
\begin{align*}
\Tsym_1^{\ell_1,\ell_2}(a,b) = (\ell_1 b - a,b) \;\;\;\;\;\text{and}\;\;\;\;\; \Tsym_2^{\ell_1,\ell_2}(a,b) = (a,\ell_2a-b).
\end{align*}
The geometric origin of these symmetries is explained in \cite[\S5]{gross2010quivers} via Theorem~\ref{thm:GPS}, by exhibiting symmetries of the corresponding Gromov--Witten invariants induced by certain birational transformations.
In particular, these form two discrete slope sequences which converge to $\xi_+^{\ell_1,\ell_2}$ and $\xi_-^{\ell_1,\ell_2}$ respectively.
It follows by Corollary~\ref{cor:f_change_formula} that, for each primitive noncolinear $\mm_1,\mm_2 \in \Z^2$ and $\ell_1,\ell_2 \in \Z_{\geq 1}$ with $\ell_1\ell_2 > 4$, the set of outgoing walls in $\ks(\calD_{\mm_1,\mm_2}^{\ell_1,\ell_2})_\min$ which lie outside of the dense region form two sequences which converge to the two boundary rays of the dense region.
For brevity, we will refer to these 
rays lying outside of the dense region as the \hl{discrete rays} of $\ks(\calD_{\mm_1,\mm_2}^{\ell_1,\ell_2})_\min$.

As for the rays inside the dense region, it was conjectured based on empirical evidence that every rational slope appears in the scattering pattern (see e.g \cite[Ex. 1.6]{gross2010tropical} and \cite[Ex. 1.15]{GHKK2018}).
In the special case $\ell_1 = \ell_2$, Gross-Pandharipande proved the above conjecture by exploiting a deep connection with quiver representation theory due to Reineke.\footnote{Namely, Reineke's theorem \cite[Thm. 2.1]{reineke2010poisson} states that the scattering coefficients $\ff_{e_1,e_2}^{\ell,\ell}(a,b)$ are determined by the Euler characteristic of a suitable moduli space of stable framed representations of
 the $\ell$-Kronecker quiver 
with dimension vector $(a,b)$.
A key fact about these moduli spaces is that they have positive Euler characteristic whenever they are empty. 
Reineke--Weist have also generalized this correspondence to the case $\ell_1 \neq \ell_2$, replacing the Kronecker quiver with the complete bipartite quiver on $\ell_1+\ell_2$ vertices (see \cite[Thm. 6.1]{reineke2013refined})} 
\begin{thm}[{\cite[\S4.7]{gross2010quivers}}]\label{thm:gp_nonzero}
  For all $\ell \in \Z_{\geq 3}$ and all primitive $(a,b) \in \Z_{\geq 1}^2$, we have
  $\coef_{\ks(\calD_{e_1,e_2}^{\ell,\ell})_\min}(z^{(a,b)})|_{t=1} \neq 0$.
\end{thm}
\NI Strictly speaking, \cite[\S4.7]{gross2010quivers} only proves $\coef_{\ks(\calD_{e_1,e_2}^{\ell,\ell})_\min}(z^{\ka(a,b)})|_{t=1} \neq 0$ for some $\ka \in \Z_{\geq 1}$, but an inspection of their argument shows that we can take $\ka = 1$.

\sss

Going beyond the case $\ell_1 = \ell_2$, the following remarkable positivity phenomenon 
was discovered by Gross--Hacking--Keel--Kontsevich is the course of constructing canonical bases for cluster algebras.

\begin{thm}[{\cite[Prop. C.13]{GHKK2018}}] For all $\ell_1,\ell_2 \in \Z_{\geq 1}$ and primitive $(a,b) \in \Z^2$, we can write
\begin{align*}
 \ff_{e_1,e_2}^{\ell_1,\ell_2}(a,b) = \prod\limits_{\ka=1}^\infty \left(1 + t^{\ka(a+b)}z^{\ka(a,b)} \right)^{c_{\ka}},
 \end{align*}
where $c_\ka \in \Z_{\geq 0}$ for all $\ka \in \Z_{\geq 1}$.
\end{thm}
By combining this positivity result with the scattering diagram deformation techniques from \cite[\S1.4]{gross2010tropical} and an inductive argument, Gr\"afnitz--Luo recently extended Theorem~\ref{thm:gp_nonzero} to the case $\ell_1 \neq \ell_2$.

\begin{thm}[{\cite[Thm. 1]{grafnitz2023scattering}}]\label{thm:GL_nonzero}
 For all $\ell_1,\ell_2 \in \Z_{\geq 1}$ with $\ell_1\ell_2 > 4$ and all primitive $(a,b) \in \Z_{\geq 1}^2$ lying in the dense region, we have $\coef_{\ks(\calD_{e_1,e_2}^{\ell_1,\ell_2})_\min}(z^{(a,b)})|_{t=1} \neq 0$.
\end{thm}

Together with the change of lattice trick as in Corollary~\ref{cor:COL_calD}, we get:
\begin{cor}
 For any primitive noncolinear $\mm_1,\mm_2 \in \Z^2$, we have $\coef_{\ks(\calD_{\mm_1,\mm_2}^{\ell_1,\ell_2})_\min}(z^\mm)|_{t=1} \neq 0$ whenever $\mm = a\mm_1 + b\mm_2$ for some $a,b \in \Z_{\geq 1}$ such that $(a,b)$ lies in the dense region of $\ks(\calD_{e_1,e_2}^{\nu(\mm_1)\ell_1,\nu(\mm_2)\ell_2})_\min$, i.e. whenever $\xi_-^{\nu(\mm_1)\ell_1,\nu(\mm_2)\ell_2} < b/a < \xi_+^{\nu(\mm_1)\ell_1,\nu(\mm_2)\ell_2}$.
\end{cor}
\NI Here as before we put $\nu(\mm_i) = \ind\left(\ann_{\Z^2}(\mm_i) \subset \ann_{\NN'}(\mm_i)\right)$, with $\NN'$ the dual lattice of $\MM' := \lan \mm_1,\mm_2\ran \subset \MM = \Z^2$,
and $\xi_{\pm}^{\bullet,\bullet}$ are as in \eqref{eq:xipm}.

\subsection{Well-placed curves from basic scattering diagrams}\label{subsec:wp_and_basic_sd}

Now suppose that $(X,\calN)$ is a uninodal Looijenga pair with a toric model $\tormod$
having data $\mm_1,\dots,\mm_J$ and $\ell_1,\dots,\ell_J$ as in Notation~\ref{not:tormod_m_and_l}, and recall that the function $\wp_X$ sets up a bijection from $\Z^2_{\geq 0} / \sim$ to $\Z^2$.
Let us further restrict to the case $J = 2$, and consider the sublattice $\lan \mm_1,\mm_2\ran \subset \Z^2$ spanned by $\mm_1,\mm_2$.
Observe that the results from \S\ref{subsec:scat_pos} together with Corollary~\ref{cor:N_logcoef_form} give a complete description of those coprime $p,q \in \Z_{\geq 1}$ for which the count
$N_{X,\calN}(p,q)$ of well-placed curves is nonzero (and hence a positive integer). 
\begin{cor}\label{cor:uninodal_rays}
Let $(X,\calN)$ be a uninodal Looijenga pair with a toric model $\tormod$ having $J=2$.
For any coprime $p,q \in \Z_{\geq 1}$, we have $N_{X,\calN}(p,q) = 0$ unless $\wp_X(p,q) \in \lan \mm_1,\mm_2\ran$. 
For $\wp_X(p,q) \in \lan \mm_1,\mm_2\ran$, we have 
$N_{X,\calN}(p,q) \neq 0$ if and only of one if the following holds:
\begin{itemize}
  \item $\wp_X(p,q)$ is a discrete ray of $\ks(\calD_\tormod)_\min$
  \item $\wp_X(p,q)$ lies in the dense region of $\ks(\calD_\tormod)_\min$.
\end{itemize}
\end{cor}

Of particular interest is the case when  $X$ is a rigid del Pezzo surface and $\tormod = \tormod_X$ is the corresponding toric model from \S\ref{subsec:rigid_dP_tor_mods}.
In this situation, recall that the ellipsoid embedding function $c_X$ contains an infinite staircase such that the $x$-values of the outer corners accumulate at a point $a_\acc^X \in \R_{>1}$.
The next lemma shows that, under the bijection $\wp_X$, the outer corners  precisely match up with the discrete part of the corresponding scattering diagram, while those $p/q$ beyond the accumulation point correspond to the dense region.

\begin{lemma}\label{lem:corners_match_sd}
Let $X$ be a rigid del Pezzo surface with its toric model $\tormod_X$ and associated data $\mm_1,\dots,\mm_J$ and $\ell_1,\dots,\ell_J$ as in \S\ref{subsec:rigid_dP_tor_mods}, and assume further that $J=2$.\footnote{In other words, $X = \CP^1 \times \CP^1$ or $X = \bl^j\CP^2$ for $j \in \{0,3,4\}$.}
For any primitive $(p,q) \in \Z_{\geq 1}^2$, we have
  \begin{itemize}
        \item $\R_{\geq 0} \cdot \wp_X(p,q)$ is a discrete ray of $\ks(\calD_{\mm_1,\mm_2}^{\ell_1,\ell_2})_\min$ if and only if $p/q$ or $q/p$ is the $x$-value of an outer corner of the infinite staircase $c_X|_{[1,a_\acc^X]}$
    \item $\wp_X(p,q)$ lies in the dense region of $\ks(\calD_{\mm_1,\mm_2}^{\ell_1,\ell_2})_\min$ if and only if $p/q$ or $q/p$ lies in $(a_\acc^X,\infty)$.
  \end{itemize}

\end{lemma}
\begin{cor}\label{cor:dense_beyond_acc}
 In the setting of Lemma~\ref{lem:corners_match_sd}, for a reduced fraction $p/q \in (a_\acc^X,\infty)$ we have $N_{X,\calN}(p,q) \neq 0$ if and only if $\wp_X(p,q) \in \lan \mm_1,\mm_2\ran$. 
\end{cor}
\NI Inspecting Table~\ref{prop:fund_bij}, $\wp_X(p,q) \in \lan \mm_1,\mm_2\ran$ is equivalent to $p+q \equiv 0 \text{ mod } 3$ in the case $X = \CP^2$, $p+q \equiv 0 \text{ mod } 2$ in the case $X = \CP^1 \times \CP^1$, and it is 
always satisfied for $X = \bl^j \CP^2$ with $j \in \{3,4\}$.

\begin{proof}[Proofs of Theorem~\ref{thmlet:sesqui_plane_curves} and Theorem~\ref{thmlet:curves_in_X}]

Theorem~\ref{thmlet:sesqui_plane_curves} and Theorem~\ref{thmlet:curves_in_X}(a) in the cases $J=2$ (in particular for $X = \CP^1 \times \CP^1$) follow immediately by combining Corollary~\ref{cor:uninodal_rays} and Lemma~\ref{lem:corners_match_sd}.
The remaining cases of Theorem~\ref{thmlet:curves_in_X} were deduced from these in \S\ref{sec:symmetries}.
\end{proof}

\begin{rmk}\label{rmk:precise_density}
The above argument actually gives a precise description of $\Sset_X$ for the rigid del Pezzo surfaces with $J=2$ strands (the case $X=\CP^2$ is already covered by Theorem~\ref{thmlet:sesqui_plane_curves}).
More precisely, for any reduced fraction $p/q$, we have:
\begin{itemize}
  \item if $X = \CP^1 \times \CP^1$ and $p/q > a_\acc^X$, there is a $(p,q)$-well-placed curve in $X$ if and only if $p+q \equiv 0\text{ mod } 2$
  \item if $X = \bl^j\CP^2$ with $j \in \{3,4\}$ and $p/q > a_\acc^X$, there is a $(p,q)$-well-placed curve in $X$.
\end{itemize}
Furthermore, using the symmetry argument in \S\ref{sec:symmetries}, we have:
\begin{itemize}
  \item if $X$ is a blowup of $\CP^2$ at $k \geq 5$ points which are either very general or lie in the smooth locus of a fixed nodal cubic, then we have $\Sset_X = [1,\infty) \cap \Q$.
\end{itemize}
It should be also possible to give sharp descriptions in the $J=3$ cases (i.e. $X = \bl^j\CP^2$ for $j=1,2$) by extending the analysis in \cite{grafnitz2023scattering} to scattering diagrams with more than two initial rays.
Incidentally, we could also apply Theorem~\ref{thm:wp_and_sd} in the reverse direction, in order to deduce structural results for certain scattering diagrams with three or more incoming rays via Theorem~\ref{thmlet:curves_in_X}.
\end{rmk}

\begin{example}
  Recall that in the case $X = \CP^2$ the accumulation point is $a_\acc^X = \tau^4 = \tfrac{7+3\sqrt{5}}{2}$.
Using the fundamental bijection in \eqref{eq:fundbijCP2},
   each reduced fraction $p/q \geq 2$ corresponds to a ray
  \[
  \R_{\geq 0} \cdot \wp_{\CP^2}(p,q) = \R_{\geq 0} \cdot (q,5q-p) = \R_{\geq 0} \cdot (1,5-p/q),
  \]
  while fractions with $p/q \le 1/2$ correspond to rays $\R_{\geq 0} \cdot (-1, 2-q/p)$.
  In particular, for $p/q$ approximating $\tau^4$, the corresponding ray approximates 
  $\R_{\geq 0} \cdot (1,5-\tau^4)$,
while, for $p/q$ approximating $1/\tau^4$, the corresponding ray approximates $\R_{\geq 0} \cdot (-1,2-\tau^4)$.

Meanwhile, as in Example~\ref{ex:nu_CP2}, the dense region of the corresponding scattering diagram $\ks(\calD_{\mm_1,\mm_2}^{1,1})_\min$ is the image of the cone spanned by $(1,\xi_\pm^{3,3})$ under the linear map $\psi: \R^2 \ra \R^2$ sending $e_1$ to $\mm_1 = (1,0)$ and $e_2$ to $\mm_2 = (-1,-3)$, 
where $\xi_\pm^{3,3} = \tfrac{1}{2}(3 \pm \sqrt{5})$ as in \eqref{eq:xipm}. 
After some arithmetic, this is the same as the cone spanned by the rays $\R_{\geq 0} \cdot (1,5-\tau^4)$ and $\R_{\geq 0} \cdot (-1,2-\tau^4)$.
In other words, under the fundamental bijection, $(p,q)$-well-placed curves with $p/q > \tau^4$ or $p/q < 1/\tau^4$ fill out two corresponding halves of the dense region of $\ks(\calD_{\mm_1,\mm_2}^{1,1})_\min$, while
$(p,q)$-well-placed curves with $p = 0$ or $q = 0$ (i.e. those intersecting $\calN$ in a single point away from the node as in Convention~\ref{conv:pq=0}) correspond to the ``central ray'' $\R_{\geq 0} \cdot (0,-1)$ of the dense region. 
\end{example}

\begin{rmk}
Under the bijection $\wp_X$, the transformations $\Phi_X,\Psi_X$ discussed in \S\ref{sec:symmetries} are closely related to the scattering symmetries $\Tsym_1,\Tsym_2$ considered in \cite[\S5]{gross2010quivers}.
In fact, the change of lattice formula in  Lemma~\ref{lem:subscatter} can be understood more geometrically in terms of finite degree toric morphisms induced by passing to finite index sublattices for the relevant fans. 
Using this perspective, it is possible to view the birational transformations underlying $\Tsym_1,\Tsym_2$ as automorphisms of the universal cover of $X \setminus \calN$, corresponding to ``twisted square roots'' of $\Phi_X,\Psi_Y$ in the sense of \cite{kollar2024cubic} (see e.g. \cite[Ex. 1]{kollar2024cubic} for the case $X = \CP^2$).
We will elaborate on the symmetries of uninodal Looijenga pairs and scattering diagrams in a followup paper.
\end{rmk}

\renewbibmacro{in:}{}
\printbibliography

\end{document}